\documentclass[a4paper,12pt, reqno]{amsart}

\usepackage{ amssymb, amsmath, enumerate, amsfonts, amsthm, mathrsfs, url, bm, mathtools}

\usepackage{times}

\usepackage{xcolor}  	
\usepackage[backref]{hyperref}
\hypersetup{
	colorlinks,
    linkcolor={blue!60!black},
    citecolor={blue!60!black},
    urlcolor={red!60!black}
}

\usepackage{color}

\numberwithin{equation}{section}

 \newcommand{\comment}[1]{}  

\usepackage[margin=1in]{geometry}

\newcommand{\Q}{\mathbb{Q}}

\newcommand{\C}{\mathbb{C}}

\newcommand{\F}{\mathbb{F}}

\newcommand{\set}[1]{\mathopen{}\left\{#1\mathclose{}\right\}}
\newcommand{\bigset}[1]{\bigl\{ #1 \bigr\}}

\newcommand{\abs}[1]{\mathopen{}\left| #1\mathclose{}\right|}
\newcommand{\bigabs}[1]{\bigl| #1 \bigr|}
\newcommand{\Bigabs}[1]{\Bigl| #1 \Bigr|}
\newcommand{\biggabs}[1]{\biggl| #1 \biggr|}

\newcommand{\sqbrac}[1]{\mathopen{}\left[ #1 \mathclose{}\right]}

\newcommand{\brac}[1]{\mathopen{}\left( #1 \mathclose{}\right)}

\newcommand{\Bigbrac}[1]{\Bigl( #1 \Bigr)}
\newcommand{\biggbrac}[1]{\biggl( #1 \biggr)}

\newcommand{\norm}[1]{\mathopen{}\left\| #1\mathclose{}\right\|}
\newcommand{\bignorm}[1]{\big\| #1 \big\|}
\newcommand{\normnorm}[1]{\| #1 \|}

\newcommand{\ang}[1]{\mathopen{}\left\langle#1\mathclose{}\right\rangle}

\newcommand{\recip}[1]{\frac{1}{#1}}
\newcommand{\trecip}[1]{\tfrac{1}{#1}}

\newcommand{\E}{\mathbb{E}}

\newcommand{\supp}{\mathrm{supp}}

\newcommand{\codim}{\mathrm{codim}\,}

\newcommand{\Spec}{\mathrm{Spec}}

\newcommand{\eps}{\varepsilon}

\makeatletter
\let\@@pmod\pmod
\DeclareRobustCommand{\pmod}{\@ifstar\@pmods\@@pmod}
\def\@pmods#1{\mkern4mu({\operator@font mod}\mkern 6mu#1)}
\makeatother

\newtheorem{theorem}{Theorem}[section]
\newtheorem{iteration}[theorem]{Iteration}

\newtheorem{conjecture}[theorem]{Conjecture}
\newtheorem{corollary}[theorem]{Corollary}
\newtheorem{proposition}[theorem]{Proposition}
\newtheorem{lemma}[theorem]{Lemma}

\theoremstyle{definition}
\newtheorem{definition}[theorem]{Definition}

\numberwithin{theorem}{section}

\renewcommand{\leq}{\leqslant}
\renewcommand{\geq}{\geqslant}

\begin{document}

\title
{An inverse theorem for the Gowers $U^3$-norm relative to quadratic level sets}

\author{Sean Prendiville}
\address{School of Mathematical Sciences\\
Lancaster University
}
\email{s.prendiville@lancaster.ac.uk}

\date{\today}

\begin{abstract}
We prove an effective version of the inverse theorem for the Gowers $U^3$-norm for functions  supported on high-rank quadratic level sets in  finite vector spaces. For configurations controlled by the $U^3$-norm (complexity-two configurations), this enables one to run a density increment argument with respect to quadratic level sets, which are analogues of Bohr sets in the context of quadratic Fourier analysis on finite vector spaces.  We demonstrate such an argument by deriving an {exponential} bound on the Ramsey number of  three-term progressions which are the same colour as their common difference (``Brauer quadruples''), a result we have been unable to establish by other means.

Our methods also yield polylogarithmic bounds on the density of sets lacking translation invariant configurations of complexity two. Such bounds for four-term progressions were obtained by Green and Tao using a simpler weak-regularity argument. In an appendix, we give an example of how to generalise Green and Tao's argument to other translation-invariant configurations of complexity two. However, this crucially relies on an estimate coming from the Croot-Lev-Pach polynomial method, which may not be applicable to all systems of complexity two. Hence running a density increment with respect to quadratic level sets may still prove useful for such problems. It may also serve as a model for running density increments on more general nil-Bohr sets, with a view to effectivising other Szemer\'edi-type theorems.
\end{abstract}

\maketitle

\setcounter{tocdepth}{1}
\tableofcontents

\section{Introduction}\label{sec:intro}
Bohr sets, as popularised in Bourgain's work on three-term progressions \cite{BourgainTriples},  have become a well-used tool in additive combinatorics. By localising to such structures, many Fourier analytic arguments become  quantitatively more efficient, at the cost of introducing extra technicalities. In the context of quadratic Fourier analysis on $\F_p^n$, the appropriate analogue of a Bohr set is the level set of a quadratic polynomial. In the pursuit of good bounds in variants of Szemer\'edi's theorem for configurations of complexity two (see \cite{GowersWolfTrue} or \cite{manners2021true} for a definition of complexity), it is useful to be able to run density increment arguments with respect to these quadratic level sets. This requires generalising Green and Tao's \cite{GreenTaoInverse} inverse theorem for the $U^3$-norm from functions supported on $\F_p^n$ to functions supported on (sparse) quadratic level sets, without the quantitative loss associated with applying the global inverse theorem to the sparsely supported function. We provide such a generalisation in this article.
\begin{definition}[$U^2$ and $U^3$-norm]\label{def:u3} Given a complex-valued function $f$ on an additively written abelian group, define the $U^2$-norm by\footnote{All summations are over the underlying group.
}
$$
\norm{f}_{U^2}^4 := \sum_{x, h_1, h_2} f(x)\overline{f(x+h_1)f(x+h_2)}f(x+h_1+h_2).
$$
Writing $\Delta_h f(x) := f(x)\overline{f(x+h)}$, define the $U^3$-norm 
$$
\norm{f}_{U^3}^{8} := \sum_h\norm{\Delta_h f}_{U^2}^4.
$$
\end{definition}
\begin{theorem}[$U^3$-inverse theorem on quadratic level sets]\label{thm:u3-inverse}
Let $H$ be a finite vector space over $\F_p$ with $p$ an odd prime, let $Q = (q_1, \dots, q_d)$ be a $d$-tuple of quadratic polynomials\footnote{For our precise definition of quadratic polynomial, see \S\ref{sec:notation}.}  $q_i : H \to \F_p$. Given a 1-bounded function $f : H \to \C$ with support $\set{x : f(x) \neq 0}$ contained in $Q^{-1}(0)$, suppose that  
$
\norm{f}_{U^3} \geq \delta \norm{1_{Q^{-1}(0)}}_{U^3}
$. Then either there exists a quadratic polynomial $q : H \to \F_p$  such that
\begin{equation}\label{eq:global-correlation}
\Bigabs{\sum_x f(x) e^{2\pi i q(x)/p}} \gg_p \delta^{O_p(1)} |Q^{-1}(0)|,
\end{equation}
or alternatively $Q$ has low rank, in that there are scalars $\lambda_1, \dots, \lambda_d$ which are not all zero, such that the homogeneous quadratic part of $\lambda_1 q_1+\dots+ \lambda_dq_d$ has rank $O_p(d+\log(2/\delta))$.
\end{theorem}
The polynomial dependence on $\delta$ in \eqref{eq:global-correlation} would not be possible without the recent resolution of the polynomial Freiman-Ruzsa (PFR) conjecture due to Gowers, Green, Manners and Tao  \cite{gowers2024martons}. One could modify our proof to avoid this, replacing PFR with the original approach of Green and Tao \cite{GreenTaoInverse}; this would deliver exponential dependence in our inverse theorem (and polynomial dependence at the cost of passing to a subspace of polynomial codimension,  as in \cite{GreenTaoInverse}).

As previously remarked, an inverse theorem such as Theorem \ref{thm:u3-inverse} enables one to run density increment arguments with respect to (high-rank) quadratic level sets.   We demonstrate such an argument by deriving an {exponential} bound on the Ramsey number of  three-term progressions which are the same colour as their common difference (``Brauer quadruples''), a result we have been unable to prove by other means. 
\begin{theorem}[Exponential upper bound for Brauer quadruples in $\F_3^n$]\label{thm:exp-brauer}
Suppose that there exists an $r$-colouring of $\F_3^n\setminus\set{0}$ with no monochromatic \textbf{Brauer quadruple}
\begin{equation}\label{eq:brauer}
x,\ y, \ x+y,\ x+2y.
\end{equation}
Then $n \ll r^{O(1)}$.
\end{theorem}
Hitherto, the best upper bound in this problem takes the form $n \ll \exp(r^{O(1)})$, which can be recovered by emulating an argument of Chapman and the author \cite{chapman2020ramsey} (where it is carried out over the integers instead of $\F_p^n$). By removing an exponential, the upper bound in Theorem \ref{thm:exp-brauer} is polynomially close to the following lower bound.
\begin{proposition}[Exponential lower bound for Brauer quadruples in $\F_3^n$]
If $n \leq r$, then there exists an $r$-colouring of $\F_3^n\setminus\set{0}$ with no monochromatic Brauer quadruple \eqref{eq:brauer}.
\end{proposition}
\begin{proof} For each $i = 1, \dots, r$, define the $i$th colour class by
\begin{equation*}
C_i := \set{x \in \F_3^n :  [x_i \neq 0]\text{ and }(\forall j > i)[x_j =0] }. \qedhere
\end{equation*}
\end{proof}

Returning to applications of our relative $U^3$-inverse theorem, one can also use it to prove the following.
\begin{theorem}[Polylogarithmic bound for translation-invariant configurations of complexity 2]\label{thm:poly-log-density}
Fix constants $c_1, c_2, c_3 \in \F_p$. If $A \subset \F_p^n$ lacks configurations of the form
$$
x, \ x+c_1y,\ x+c_2y,\ x+c_3y, \quad(y \neq 0)
$$
then the density of $A$ satisfies $|A|/p^n \ll_{p} n^{-\Omega_{p}(1)}$.
\end{theorem}

Such a result has been proved in the special case of four-term progressions (when $c_i = i$) by Green and Tao \cite{GreenTaoNewIa}, whose approach does not rely on a relative $U^3$-inverse theorem. Instead, Green and Tao combine a weak regularity argument with  a positivity property specific to four-term progressions (see the paragraph preceding \cite[Corollary 5.4]{GreenTaoNewIa} or the footnote \cite[p.6]{GreenTaoNewIII}). Although our relative $U^3$-inverse theorem \textit{can}  be used to prove Theorem \ref{thm:poly-log-density}, this application does not fully demonstrate the utility of working relative to quadratic level sets since, in Appendix \ref{sec:density}, we show how to generalise Green and Tao's method  to give a simpler proof of Theorem \ref{thm:poly-log-density}. This adaptation involves replacing the positivity property associated to four-term progressions with an application of the Croot-Lev-Pach polynomial method \cite{croot2017progression,ellenberg2017large}. This is the first argument we are aware of which demonstrates how the polynomial method is of use in the study of configurations of complexity two (quadratic Fourier analysis). 

Present limitations of the polynomial method\footnote{Namely, the polynomial method cannot presently handle an arbitrary translation-invariant system of complexity one.} mean that the following conjecture, which is an attempt at formulating Theorem \ref{thm:poly-log-density} in maximal generality, is out of reach of  Green and Tao's method as detailed in Appendix \ref{sec:density}. However, we believe this conjecture {is} provable on combining remarkable work of Manners \cite{manners2021true}  with the density increment approach of the present paper, an approach which makes essential use of our relative $U^3$-inverse theorem. 

\begin{conjecture}[Polylog bound for \text{any} translation-invariant configuration of complexity 2]\label{conj:complex-2}
Let $\phi_1$, \dots, $\phi_k$ be a translation-invariant system (see Definition \ref{def:translation-invariant}) of $\F_p$-linear maps $\F_p^n \to \F_p^n$  of complexity at most two (see Definition \ref{def:complex-two}). Let $\psi_1, \dots, \psi_s$ denote another system of non-zero $\F_p$-linear maps $\F_p^n \to \F_p^n$. Suppose that $A \subset \F_p^n$ has the following property: there exist constants $c_1, \dots, c_s \in \F_p^n$ such that if $x_1, \dots, x_d \in \F_p^n$ satisfy
$$
\phi_1(x_1, \dots, x_d),\ \dots,\ \phi_k(x_1, \dots, x_d)\in A
$$
then this configuration is ``$\psi$-trivial'' in that some $\psi_i(x_1, \dots, x_d)$ equals $c_i$. Then the density of $A$ satisfies $|A|/p^n \ll_{p,d,k,s} n^{-\Omega_{p,d,k,s}(1)}$.
\end{conjecture}

In the process of proving our $U^3$-inverse theorem relative to quadratic level sets (Theorem \ref{thm:u3-inverse}), we prove the following $U^2$-inverse theorem relative to pseudorandom sets. A quantitative variant of this result was used by Peluse (see \cite[\S6]{peluse2022subsets}), who proved it via the transference principle and thereby obtained weaker bounds. Peluse \cite{peluse2022bcc} posed the problem of obtaining  a direct Fourier-analytic proof of this fact, given that such an approach is likely more quantitatively efficient. This we accomplish in the following.
\begin{theorem}[$U^2$-inverse theorem on  Fourier uniform sets]\label{thm:u2-inverse}
Let $G$ be a finite abelian group. Suppose that $\Phi \subset G$ is a set of density $\delta := |\Phi|/|G|$ which satisfies the Fourier uniformity estimate 
$$
\max_{\gamma \neq 1} \Bigabs{\sum_{x \in \Phi} \gamma(x)} \leq \delta^{2} |\Phi|,
$$ 
where the maximum is taken over all non-trivial characters on $G$ (homomorphisms to the multiplicative group of complex numbers on the unit circle, a set denoted $\hat G$).
Then for any 1-bounded function $f : G \to \C$ with support $\set{x : f(x) \neq 0}$ contained in $ \Phi$, we have the implication:
$$
 \norm{f}_{U^2} \geq \eta \norm{1_\Phi}_{U^2} \quad \implies \quad \exists \gamma \in \hat G \quad \Bigabs{\sum_{x\in \Phi} f(x)\gamma(x)} \gg \eta^4 |\Phi|.
$$
\end{theorem}
\subsection{Paper organisation}
%
We prove Theorem \ref{thm:u2-inverse} in \S\ref{sec:u2-inverse}, where we establish various properties of Fourier uniform sets. In \S\ref{sec:high-rank} we state and prove a number of results concerning high-rank quadratic level sets; results we use throughout the remainder of the paper. In \S\S\ref{sec:many-partial}, \ref{sec:many-genuine}, \ref{sec:sym} we carry out the proof of our relative $U^3$-inverse theorem, which follows the strategy of Gowers \cite{GowersNewFour} and Green-Tao \cite{GreenTaoInverse}, but with a number of new technical ideas required to work relative to quadratic level sets. In \S\S\ref{sec:ramsey-bound}, \ref{sec:part-high}, \ref{sec:sparsity-expansion-density-increment} we demonstrate how to run density increment arguments relative to quadratic level sets,  bounding the Ramsey number of Brauer quadruples. In Appendix \ref{sec:density}, we show how to generalise an argument  of Green and Tao \cite{GreenTaoNewIa} in order to prove Theorem \ref{thm:poly-log-density}

\subsection{Acknowledgements} Part of this work was carried out whilst  participating in the workshop \textit{``High-dimensional phenomena in discrete analysis''} at the American Institute of Mathematics, whose hospitality the author gratefully acknowledges. Thanks to  James Leng, Freddie Manners, Sarah Peluse, Mehtaab Sawhney and Terence Tao for informative conversations. Thanks in particular to Tamar Ziegler for their insight and encouragement.

\subsection{Notation}\label{sec:notation}
 Given an abelian group $G$, let $\hat{G}$ denote the \textit{dual group} of \textit{characters} (homomorphisms from $G$ to the multiplicative group of complex numbers on the unit circle). For  a function $f : G \to \C$ and $\gamma \in \hat{G}$ we define the \textit{Fourier transform} of $f$ at $\gamma$ to be the inner product of $f$ with $\gamma$. We normalise inner products on $G$ so that this is
\begin{equation}\label{eq:fourier-transform}
\hat{f}(\gamma) :=  \sum_{x \in G} f(x) \overline{\gamma(x)}.
\end{equation}
Notice that with this normalisation $\norm{f}_2 :=\brac{ \sum_{x \in G} |f(x)|^2}^{1/2}$ and more generally $\norm{f}_q := \brac{ \sum_{x \in G} |f(x)|^q}^{1/q}$ for $q \in (0, \infty)$.
We use the uniform probability measure on $\hat{G}$, so that 
\begin{equation}\label{eq:Lp-norm-dual}
\normnorm{\hat{f}}_{q} := \brac{\E_{\gamma \in \hat{G}} |\hat{f}(\gamma)|^q}^{\recip{q}}.
\end{equation}
Here the operator $\E$ is defined, for a finite  non-empty set $S$ and function $g:S\to\C$, via
\[
\E_{x\in S}\ g(x):=\frac{1}{|S|}\sum_{x\in S}g(x).
\]
We define the support of $f : G \to \C$ to be
\begin{equation}\label{eq:supp}
\supp(f):= \set{x \in G  : f(x) \neq 0}.
\end{equation}
We say that $f$ is \emph{1-bounded} if the modulus of the function does not exceed 1.

Throughout $H$ will be used to denote a finite vector space over $\F_p$ where $\F_p$ is the field with $p$ elements and $p$ is an odd prime. In this case, there is a group isomorphism between $\hat{H}$ and the \textit{dual vector space} $H^*$ which consists of all $\F_p$-linear maps $\ell : H \to \F_p$ (which we also call \textit{linear forms}). One such isomorphism takes the linear form $h \mapsto \ell h$ to the character $\gamma_\ell : h \mapsto e_p(-\ell h)$.  We will often abuse notation and write $\hat{f}(\ell)$ for $\hat{f}(\gamma_\ell)$, so that
$$
\hat{f}(\ell) = \sum_{x \in H} f(x)e_p(\ell x).
$$

A \textit{quadratic polynomial} on $H$ is a map of the form $x \mapsto b(x,x) + \ell x + c$, where $b: H\times H \to \F_p$ is a bilinear form, $\ell : H \to \F_p$ is a linear form and $c \in \F_p$ is a constant. The \textit{homogeneous quadratic part} of this polynomial is $b(x,x)$, the \textit{homogeneous linear part} is $\ell x$ and an \textit{associated bilinear form} $(x,y) \mapsto \tilde b(x,y)$ satisfies $\tilde b(x,x) = b(x,x)$ for all $x \in H$. When $p$ is odd, we may assume that any quadratic polynomial $x \mapsto b(x,x) + \ell x + c$ has a unique associated bilinear form which is  \textit{symmetric}, so that $b(x,y) = b(y,x)$ for all $x, y \in H$.  A \textit{(homogeneous) quadratic form} is a quadratic polynomial with zero homogeneous linear part and zero constant term.  For a bilinear form $b : H \times H \to \F_p$ and fixed $x \in H$, we have $b(x, \cdot) \in H^*$ (the dual vector space defined in the previous paragraph);  to save on parentheses, let us denote this element of $H^*$ by $x^Tb$. Similarly, let $by$ denote the element of $H^*$ given by $b(\cdot,y)$. Notice that $b$ is symmetric if and only if $x^Tb = bx$ for all $x \in H$. Given a $d$-tuple $B = (b_1, \dots, b_d)$ of  bilinear forms on $H$, we write $B(x,y)$ or $x^TBy$ for the $d$-tuple $(x^Tb_1y, \dots, x^Tb_dy) \in \F_p^d$. Similarly, define the $d$-tuple $x^TB= (x^Tb_1, \dots,  x^Tb_d) \in (H^*)^d$ and $By = (b_1y, \dots, b_dy) \in (H^*)^d$. For $L = (\ell_1, \dots, \ell_d) \in (H^*)^d$ we write $\ang{L} = \ang{\ell_1, \dots, \ell_d}$ for the $\F_p$-span of the $d$ elements in the coordinates of $L$; hence $\ang{L}$ forms a subspace of $H^*$ of dimension at most $d$. Given a subspace $U$ of $H$ write
$
U^\perp := \set{\ell \in H^* : \ell u  = 0 \ (\forall u \in U)}
$. By rank-nullity, this forms a subspace of $H^*$ of codimension $\dim U$. Similarly, given a subspace $V$ of $H^*$ write
$
V^\perp := \set{h \in H : \ell h  = 0 \ (\forall \ell \in V)},
$ a subspace of $H$ of codimension $\dim V$.
Since $V \subset (V^\perp)^\perp$ and these spaces have the same finite dimension, we see that $(V^\perp)^\perp = V$.

For a function $f$ and positive-valued function $g$, write $f \ll g$ or $f = O(g)$ if there exists a constant $C$ such that $|f(x)| \le C g(x)$ for all $x$. We write $f = \Omega(g)$ (or $f \gg g$) if $g\ll f$.  If $f\ll g$ and $f\gg g$ then we write $f \asymp g$. Possible dependence of implicit constants on other parameters is indicated with subscripts.

\section{A restriction estimate and $U^2$-inverse theorem on Fourier uniform sets}\label{sec:u2-inverse}
Key to our approach to the relative $U^3$-inverse theorem (Theorem \ref{thm:u3-inverse}) is the fact that high-rank quadratic level sets are pseudorandom with respect to the $U^2$-norm (``Fourier uniformity''). In this section we establish  certain properties of Fourier uniform sets which are used repeatedly in the remainder of the paper. Namely, for functions defined on such sets: we estimate their spectra, show that they satisfy a restriction estimate and prove the $U^2$-inverse theorem previously advertised (Theorem \ref{thm:u2-inverse}).
\begin{definition}[Fourier uniform set]\label{def:fourier-uniform}
Given a finite abelian group $G$ and $\Phi \subset G$, we say that $\Phi$ is \textit{$\eps$-Fourier uniform} on $G$ if
$$
\bignorm{\hat{1}_\Phi - \tfrac{|\Phi|}{|G|} \hat{1}_G}_\infty \leq \eps \normnorm{\hat{1}_\Phi}_\infty.
$$
Here we write $\norm{f}_\infty$ for $\sup_\gamma |f(\gamma)|$, which is a maximum when the underlying domain is finite. We note that (with our notational conventions) $\normnorm{\hat{1}_\Phi}_\infty = |\Phi|$ and 
$$
\hat{1}_G(\gamma) = \begin{cases} |G| & \text{ if $\gamma$ is the trivial character;}\\
						0 & \text{ otherwise.}\end{cases}
$$
Hence an equivalent formulation of $\eps$-Fourier uniformity in this context is that
$$
\max_{\gamma\neq 1} \Bigabs{\sum_{x \in \Phi} \gamma(x)} \leq \eps |\Phi|,
$$
which for $G= H \leq \F_p^n$ can be written as
$\displaystyle
\max_{h\in H\setminus\set{ 0}} \Bigabs{\sum_{x \in \Phi} e_p(hx)} \leq \eps |\Phi|.
$
\end{definition}
\begin{lemma}[Spectral estimate]\label{lem:spec-estimate}
Let $G$ be a  finite abelian group and let $\Phi \subset G$ be $\eps$-Fourier uniform (see Definition \ref{def:fourier-uniform}). For $f : G \to \C$ define the large spectrum
$$
\Spec(f, K) := \bigset{\gamma \in \hat{G} : \bigabs{\hat{f}(\gamma)}\geq K}
$$
and let us normalise $L^2$ with counting measure
\begin{equation}\label{eq:L2}
\norm{f}_2 := \Bigbrac{\sum_{x\in G} |f(x)|^2}^{1/2}.
\end{equation}
If $\supp(f) \subset \Phi$, then for any $K \geq (2 \eps |\Phi|)^{1/2} \norm{f}_2$ we have
$$
\abs{\Spec(f,K )} \leq \frac{2|\Phi|\norm{f}_2^{2}}{K^2}.
$$
In particular, if $f$ is 1-bounded then for any $\delta \geq (2\eps)^{1/2}$ we have
$
\abs{\Spec(f, \delta |\Phi|)} \leq 2 \delta^{-2}.$
\end{lemma}
When $f$ is 1-bounded with support in $\Phi$, Parseval's identity immediately yields that 
$$|\Spec(f, \delta|\Phi|)| \leq \delta^{-2} |G|/|\Phi|.$$ The utility in the spectral estimate is that, when $\Phi$ is pseudorandom, we can remove the reciprocal of the density of $\Phi$ from the Parseval bound.
\begin{proof}
Write $S := \Spec(f,K)$. Then we have that
$
K \abs{S} \leq \sum_{\gamma  \in S} \bigabs{\hat{f}(\gamma)}
$. 
Hence there exists a 1-bounded function $\omega : S \to \C$ with
\begin{multline*}
K \abs{S} \leq \sum_{\gamma  \in S} \hat{f}(\gamma) \omega(\gamma) = \sum_{x\in \Phi} f(x)\sum_{\gamma\in S}\overline{\gamma(x)}\omega(\gamma)
	\leq \norm{f}_2\Bigbrac{\sum_{x\in \Phi}\sum_{\gamma_1, \gamma_2 \in S} \overline{\gamma_1}\gamma_2(x)\omega(\gamma_1)\overline{\omega(\gamma_2)}}^{1/2} .
\end{multline*}
Therefore
$$
K^2 \abs{S}^2 \leq  \norm{f}_2^2\sum_{\gamma_1,\gamma_2\in S}\abs{\hat{1}_\Phi(\gamma_1\overline{\gamma_2})}.
$$
We break the above sum into the diagonal contribution (when $\gamma_1 = \gamma_2$) and the remainder (when $\gamma_1\overline{\gamma_2}$ is a non-constant character). For the latter we use our $\eps$-Fourier uniformity estimate to deduce that
$$
K^2 \abs{S}^2 \leq  \norm{f}_2^2 \brac{|S||\Phi| + |S|^2\eps|\Phi| }.
$$
Re-arranging, this gives
$
K^2|S| \leq  \norm{f}_2^2|\Phi|\brac{1 +\eps |S|}.
$
\end{proof}
\begin{theorem}[Restriction estimate on Fourier uniform sets]\label{thm:restriction}
Let $G$ be a finite abelian group and let $\Phi \subset G$ be $\eps$-Fourier uniform (see Definition \ref{def:fourier-uniform}). Then for any $q > 2$ and any $f : G \to \C$ with $\supp(f) \subset \Phi$ we have that\footnote{Regarding normalisations, see \eqref{eq:fourier-transform} for the Fourier transform and \eqref{eq:L2} for the $L^2$-norm.}
$$
\E_{\gamma \in \hat{G}} |\hat{f}(\gamma)|^q \ll_q \norm{f}_2^q|\Phi|^{\frac{q}{2}}\brac{ |G|^{-1} + \eps^{\frac{q}{2}-1}|\Phi|^{-1}}.
$$
In particular, if $f$ is 1-bounded and $\eps \leq \brac{|\Phi|/|G|}^{\frac{2}{q-2}}$ (so that the level of Fourier uniformity of $\Phi$ is sufficiently small in terms of its density),  then 
$$
\E_{\gamma \in \hat{G}} \bigabs{\hat f(\gamma)}^q \ll_q |\Phi|^q|G|^{-1}. 
$$
\end{theorem}
\begin{proof}
Re-normalising, suppose that $\norm{f}_2 \leq 1$. With this normalisation, the Cauchy-Schwarz inequality yields that
$
\normnorm{\hat f}_\infty \leq |\Phi|^{1/2}
$, hence we may divide our sum over $\hat{G}$ into the following dyadic spectra 
$$
\E_\gamma \bigabs{\hat{f}(\gamma)}^q = |G|^{-1} \sum_{j=1}^J \sum_{\frac{|\Phi|^{1/2}}{2^{j}} < |\hat{f}(\gamma)| \leq 2\frac{|\Phi|^{1/2}}{2^j} } \bigabs{\hat{f}(\gamma)}^q  + |G|^{-1}\sum_{ |\hat{f}(\gamma)| \leq \frac{|\Phi|^{1/2}}{2^J} } \bigabs{\hat{f}(\gamma)}^q.
$$
Notice that if $2^{2j+1} \leq 1/\eps$ then our spectral estimate (Lemma \ref{lem:spec-estimate}) gives that
$$
\bigabs{\bigset{\gamma : |\Phi|^{1/2}2^{-j} \leq |\hat{f}(\gamma)|}} \leq 2^{2j+1}.
$$
Choosing $J$ to be maximal so that $2^{2J+1}\leq 1/\eps $, we obtain
$$
\sum_{j=1}^J \sum_{\frac{|\Phi|^{1/2}}{2^{j}} \leq |\hat{f}(\gamma)| \leq \frac{2|\Phi|^{1/2}}{2^j} } \bigabs{\hat{f}(\gamma)}^q \leq \sum_{j=1}^J 2^{1+q}|\Phi|^{\frac{q}{2}} 2^{(2-q)j} \ll_q |\Phi|^{\frac{q}{2}} \sum_{j=1}^\infty  2^{(2-q)j} \ll_q |\Phi|^{\frac{q}{2}}.
$$
By maximality of $J$ we have 
$
2^{-(2(J+1)+1)} < \eps
$, equivalently $2^{-J} < (8\eps)^{1/2}$, hence
$$
\sum_{ |\hat{f}(\gamma)| \leq \frac{|\Phi|^{1/2}}{2^J} } \bigabs{\hat{f}(\gamma)}^q \leq (8\eps|\Phi|)^{\frac{q-2}{2}} \sum_{ \gamma } \bigabs{\hat{f}(\gamma)}^2 \leq (8\eps|\Phi|)^{\frac{q-2}{2}}  |G|.
$$
Putting all of this together gives 
$
\E_\gamma \bigabs{\hat{f}(\gamma)}^q \ll_q |\Phi|^{\frac{q}{2}}|G|^{-1} + \brac{\eps|\Phi|}^{\frac{q-2}{2}}. 
$
\end{proof}
The above restriction estimate (Theorem \ref{thm:restriction}) allows us to prove a $U^2$-inverse theorem relative to pseudorandom sets, as claimed  in Theorem \ref{thm:u2-inverse}. Since $\norm{f}_{U^2} = \normnorm{\hat{f}}_4$, this follows on taking $q = 4$ in the following more general result.
\begin{theorem}[Inverse theorem on Fourier uniform sets]\label{thm:restriction-inverse}
Let $G$ be a finite abelian group. Fix $q > 2$ and suppose that $\Phi \subset G$ is a set of density $\delta := |\Phi|/|G|$ which satisfies the Fourier uniformity estimate\footnote{We have not optimised exponents in this statement - they can likely be improved  if relevant to the application.}
$$
\max_{\gamma \in \hat G\setminus\set{ 1}} \Bigabs{\sum_{x \in \Phi} \gamma(x)} \leq \delta^{\frac{4}{q-2}} |\Phi|.
$$ 
Then for any 1-bounded function $f : G \to \C$ with $\supp(f) \subset \Phi$ we have the bound
\begin{equation}\label{eq:qth-moment}
\E_\gamma \bigabs{\hat f(\gamma)}^{q}  \ll_q \normnorm{\hat{f}}_\infty^{\frac{q-2}{2}}|\Phi|^{\frac{q+2}{2}}|G|^{-1}.
\end{equation}
In particular
$$
\frac{\normnorm{\hat{f}}_{q}}{ \normnorm{\hat{1}_\Phi}_q}\ll_q \biggbrac{\frac{\normnorm{\hat f}_\infty}{\normnorm{\hat{1}_\Phi}_\infty}}^{\frac{q-2}{2q}}.
$$
\end{theorem}
\begin{proof}
Set $q' = \frac{2+q}{2} \in (2, q)$. Then by the restriction estimate (Theorem \ref{thm:restriction}), we have that
$$
\E_\gamma \bigabs{\hat f(\gamma)}^{q'}  \ll_q |\Phi|^{q'}|G|^{-1},
$$
provided that the level of Fourier uniformity of $\Phi$ is at most $\delta^{\frac{2}{q'-2}}= \delta^{\frac{4}{q-2}}$ (which we are indeed assuming). Extracting the $q-q' = (q-2)/2$ power of the largest Fourier coefficient from the $q$th moment  gives \eqref{eq:qth-moment}. Finally, by restricting to the trivial Fourier coefficient, we observe that
\begin{equation*}
\E_\gamma \abs{\hat{1}_\Phi(\gamma)}^q \geq |\Phi|^q|G|^{-1}.
\end{equation*}
Comparing upper and lower bounds gives the final claimed conclusion.
\end{proof}

\section{High-rank quadratic level sets}\label{sec:high-rank}
In this section we establish a number of properties of high-rank quadratic level sets, all of which will be useful in the proof of our relative $U^3$-inverse theorem (Theorem \ref{thm:u3-inverse}). In particular, we establish that these level sets are pseudorandom subsets of their ambient vector space.
\begin{definition}[Rank]
Let $H$ be a finite dimensional vector space. We say that a $d$-tuple of symmetric bilinear forms $B= (b_1, \dots, b_d)$ on $H$ has \textit{rank} at least $R$ if, for any scalars $\lambda_1, \dots, \lambda_d$ which are not all zero, the symmetric bilinear form $\lambda_1 b_1+\dots+ \lambda_db_d$ has rank at least $R$ (so that for any basis $v_1, \dots, v_n$ of $H$, the matrix with $ij$-entry $b(v_i, v_j)$ has rank at least $R$).  When working over a field of characteristic not equal to 2, we say that a $d$-tuple of homogeneous quadratic forms $Q = (q_1, \dots, q_d)$ has rank at least $R$ if the associated $d$-tuple of symmetric bilinear forms has rank at least $R$. When the coordinates of $Q$ are quadratic polynomials, as opposed to homogeneous quadratic forms, then we say that $Q$ has rank at least $R$ if the $d$-tuple of homogeneous quadratic parts has rank at least $R$.
\end{definition}

\begin{lemma}[Rank decrement on passing to subspaces]\label{lem:rank-decrement}
Let $H$ be a finite dimensional vector space and suppose that the $d$-tuple $B= (b_1, \dots, b_d)$ of  bilinear forms on $H$ has {rank} at least $R$. Then for any subspace $V \leq H$ of codimension $D$, the rank of the restriction of $B$ to $V$ is at least $R-2D$.
\end{lemma}
\begin{proof}
Extend a basis $v_1, \dots, v_m$ of $V$ to a basis $v_1, \dots, v_{m+D}$ of $H$. Given scalars $\lambda_1, \dots, \lambda_d$, not all zero, write $b_\lambda := \lambda_1b_1 + \dots + \lambda_d b_d$.  Let $A_H$ denote the matrix of $b_\lambda$ with respect to the basis $v_1, \dots, v_{m+D}$. Let $A_V$ denote the matrix of the restriction $b_\lambda |_{V\times V}$ with respect to the basis $v_1, \dots, v_m$. Then the rank of $b_\lambda$ over $V$ (respectively over $H$) is the rank of $A_V$ (respectively $A_H$). We observe that $A_V$ can be obtained from $A_H$ by first deleting the bottom $D$ rows, then deleting the last $D$ columns. Deleting $D$ rows from $A_H$ leaves an $m\times (m+D)$ matrix with at least $R-D$ linearly independent rows. Since row-rank equals column-rank this matrix has at least $R-D$ linearly independent columns. Therefore deleting the last $D$ columns from this matrix leaves at least $R-2D$ linearly independent columns.
\end{proof}
\begin{lemma}[Weyl bound]\label{lem:weyl-bound}
Let $H $ be a finite vector space over $\F_p$ with $p \neq 2$ and let $q : H \to \F_p$ be a quadratic polynomial on $H$ whose quadratic homogeneous part has rank at least $R$. For any subspace $V \leq H$ of codimension $D$ and translate $x_0\in H$ we have
$$
\Bigabs{\sum_{x \in V+x_0} e_p\sqbrac{q(x)}} \leq |V|p^{D-R/2}.
$$
\end{lemma}
\begin{proof}
Write $b$ for a symmetric bilinear form on $H$ for which $q(x) - b(x,x)$ is a linear polynomial. By the rank-decrement estimate (Lemma \ref{lem:rank-decrement}), the set of $h \in V$ for which $b(h, \cdot)$ is the zero linear form is a  subspace of $V$ of size at most $|V|p^{2D - R}$.  The square of the absolute value we are estimating can therefore be bounded as follows:
\[
\sum_{h\in V} \sum_{x\in V} e_p\sqbrac{q(x+h+x_0) - q(x+x_0)} \leq \sum_{h\in V}\Bigabs{ \sum_{x\in V} e_p(2b(h,x)} =  \sum_{\substack{h\in V\\ 2b(h,\cdot) = 0}} |V| \leq |V|^2p^{2D-R}.\qedhere
\]
\end{proof}

\begin{lemma}[Structure of differenced level sets]\label{lem:struct-diff}
Let $H$ be a finite vector space over $\F_p$ with $p$ odd and let $Q = (q_1, \dots, q_d)$ be a $d$-tuple of quadratic polynomials on $H$ of rank at least $R$. 
For any $u \in H^s$ and $v \in H^t$ set 
$$
\Phi_{u,v} := Q^{-1}(0) \cap \bigcap_{i=1}^s\sqbrac{Q^{-1}(0)-u_i}\cap \bigcap_{j=1}^t\sqbrac{v_j - Q^{-1}(0)}.
$$
Write $B$ for the $d$-tuple of symmetric bilinear forms associated to $Q$ and $L $ for the $d$-tuple of homogeneous linear parts. Set
\begin{equation}\label{eq:diff-subspace}
V_{u,v} := \set{x \in H : B(u_i,x) = B(v_j,x)+Lx = 0 \quad (1\leq i \leq s,\ 1 \leq j\leq t)}.
\end{equation}
Then for any $x_0 \in \Phi_{u,v}$ we have 
\begin{equation}\label{eq:diff-set-formula}
\Phi_{u,v} = Q^{-1}(0) \cap \brac{V_{u,v}+x_0}.
\end{equation}
\end{lemma}

\begin{proof}
Expanding the quadratics $Q(u_i+x)$ and $Q(v_j-x)$, we  observe that $x \in \Phi_{u,v}$ if and only if
\begin{multline*}
Q(x) = 0, \quad 2B(u_i,x) = Q(0)-Q(u_i) \quad (1\leq i \leq s),\\ 2\sqbrac{B(v_j,x)+Lx} = Q(v_j) -Q(0)\quad (1\leq j \leq t).
\end{multline*}
Hence for fixed $x_0 \in \Phi_{u,v}$, we have $x \in \Phi_{u,v} - x_0$ if and only if 
\begin{equation*}
Q(x+x_0) = 0, \quad B(u_i,x) = 0 \quad (1\leq i \leq s),\quad B(v_j,x)+Lx = 0\quad (1\leq j \leq t).\qedhere
\end{equation*}
\end{proof}

\begin{lemma}[Size of a generic differenced subspace]\label{lem:size-of-diff-sub}
Let $H$ be a finite vector space over $\F_p$ with $p$ odd and let $B = (b_1, \dots, b_d)$ be a $d$-tuple of  bilinear forms on $H$ of rank (as a $d$-tuple) at least $R$. Let $L_1 = (\ell_{11}, \dots, \ell_{1d})$, \dots, $L_k = (\ell_{k1}, \dots, \ell_{kd})$ be $d$-tuples of linear forms on $H$. Then for all but $|H|^kp^{kd-R}$ tuples $ (h_1, \dots, h_k) \in H^k$, the linear map $x \mapsto (B(h_1,x) + L_1x, \dots, B(h_k,x) + L_kx)$ has full-rank, so that the maps $h_j^Tb_i + \ell_{ij}$ ($1\leq i \leq d$, $1\leq j\leq k$) are linearly independent and the subspace
$$
V_{h_1, \dots, h_k} := \set{x \in H : B(h_1, x)+L_1x = \dots = B(h_k, x)+L_kx = 0 },
$$
has codimension $kd$.
\end{lemma}
\begin{proof}
If the maps $h_j^Tb_i+\ell_{ij}$ ($1\leq i \leq d$, $1\leq j\leq k$) are not linearly independent, then there exists $\lambda_{ij}$ not all zero such that $\sum_{i,j} \lambda_{ij} (h_j^Tb_i+\ell_{ij}) = 0$. Let $j_0$ be minimal for which  $\lambda_{1j_0}$, \dots, $\lambda_{dj_0}$ are not all zero. Then
$$
h_{j_0}^T\Bigbrac{\sum_{i=1}^d \lambda_{i{j_0}}b_i} = -\sum_{i=1}^d \lambda_{ij_0}\ell_{ij}- \sum_{i=1}^d \sum_{j\neq {j_0}}  \lambda_{ij}(h_j^Tb_i+\ell_{ij}).
$$
Fixing a choice of each $\lambda_{ij}$ (at most $p^{kd}$ choices) and choices for $h_j$ with $j \neq j_0$ (at most $|H|^{k-1}$ choices), our rank assumption on $B$ implies that at most $|H|p^{-R}$ choices remain for $h_{j_0}$. The result now follows from the union bound.
\end{proof}
\begin{lemma}[Pseudorandomness of quadratic level sets on cosets]\label{lem:pseud-quad-diff}
Let $H$ be a finite vector space over $\F_p$ with $p$ odd and let $Q = (q_1, \dots, q_d)$ be a $d$-tuple of quadratic polynomials on $H$ of rank at least $R$. Let $V\leq H$ be a subspace of codimension $D$ and let $h \in H$. Then $Q^{-1}(0)$ is Fourier uniform on the coset $V+h$ in the following sense:  if $\ell \in H^*$  then\footnote{We recall from \S\ref{sec:notation} that $V^\perp = \set{\ell \in H^*: \ell v = 0\ \forall v \in V}$.} 
$$
\biggabs{\sum_{x \in Q^{-1}(0)\cap (V+h)} e_p(\ell x) -
1_{V^\perp}(\ell)e_p(\ell h) |V|p^{-d}} \leq |H|p^{-R/2} . 
$$
\end{lemma}

\begin{proof}
By the orthogonality relations
\begin{equation*}
\sum_{x \in Q^{-1}(0)\cap (V+h)} e_p(\ell x)
		= \E_{\lambda \in \F_p^d} \sum_{x \in V+h} e_p\brac{\lambda Q(x)+ \ell x}.
\end{equation*}
By the Weyl bound (Lemma \ref{lem:weyl-bound}),  the inner sum is at most $|V|p^{D-R/2} = |H|p^{-D/2}$ when $\lambda \in \F_p^d\setminus\set{0}$. When $\lambda = 0$, the inner sum becomes $e(\ell h)\sum_{x \in V} e_p\brac{ \ell x}$, which is zero if $\ell \notin V^\perp$ and equals $e(\ell h)|V|$ otherwise.
\end{proof}

\begin{corollary}[Size of quadratic level sets and relative $U^2$-inverse theorem]\label{cor:quad-level-size}
Let $H$ be a finite vector space over $\F_p$ with $p$ odd and let $Q = (q_1, \dots, q_d)$ be a $d$-tuple of quadratic polynomials on $H$ of rank at least $R$. Then
$$
|Q^{-1}(0)| = |H|p^{-d} + O\brac{|H|p^{-R/2}},
$$
so that either $R \ll d$, or $|Q^{-1}(0)| \asymp |H|p^{-d}$. Furthermore, either $R \ll d$ or for any 1-bounded function $f : H \to \C$ with $\supp(f) \subset Q^{-1}(0)$ we have the inequality
\begin{equation}\label{eq:u2-quad-0}
\normnorm{f}_{U^2}^4 \ll \normnorm{\hat f}_\infty |H|^2p^{-3d}.
\end{equation}

\end{corollary}

\begin{proof}
Take $V = H$ and $h = 0$ in Lemma \ref{lem:pseud-quad-diff} to yield 
$$
\biggabs{\sum_{x \in Q^{-1}(0)} e_p(\ell x) -
1_{\ell =0} |H|p^{-d}} \leq |H|p^{-R/2} . 
$$
This gives our asymptotic formula for $|Q^{-1}(0)|$ on taking $\ell = 0$. In addition, it tells us that either $R \ll d$ or  $Q^{-1}(0)$ is sufficiently pseudorandom for Theorem \ref{thm:restriction-inverse} to be applicable with $q = 4$, giving \eqref{eq:u2-quad-0}.
\end{proof}

\begin{corollary}[Pseudorandomness of differenced quadratic level sets]\label{cor:pseud-once-diff}
Let $H$ be a finite vector space over $\F_p$ with $p$ odd and let $Q = (q_1, \dots, q_d)$ be a $d$-tuple of quadratic polynomials on $H$ of rank at least $R$. Let $B$ denote the $d$-tuple of symmetric bilinear forms associated to $Q$ and $L$ denote the $d$-tuple of homogeneous linear parts of $Q$.  Given $h \in H$ define 
$$
 \Phi_h:= Q^{-1}(0)\cap \sqbrac{Q^{-1}(0)-h} \quad \text{and} \quad \Psi_h := Q^{-1}(0)\cap \sqbrac{h-Q^{-1}(0)}.
$$
Then, fixing  $x_0 \in \Phi_h$ and $y_0 \in \Psi_h$, for all but $ |H|p^{d - R}$ values of $h \in H$ we have both of the following\footnote{Recall (see \S\ref{sec:notation}) that $\ang{L}$ denotes the space of linear forms which are in the $\F_p$-span of the coordinates of $L$.}
\begin{align}
\Bigabs{\sum_{x \in \Phi_h} e_p(\ell (x-x_0)) - |H|p^{-2d}1_{\ell \in \ang{h^TB}}} & \leq  |H|p^{ - R/2};\label{eq:pseud-once-diff}\\
\Bigabs{\sum_{x \in \Psi_h} e_p(\ell(x-y_0)) - |H|p^{-2d}1_{\ell \in \ang{h^TB+L}}} & \leq  |H|p^{- R/2}.\label{eq:pseud-once-diff-minus}
\end{align}
In particular, either $R \ll d$ or for all but $ |H|p^{O(d) - \Omega(R)}$ values of $h \in H$, any 1-bounded function $f : H \to \C$ with $\supp(f) \subset \Phi_h$ satisfies
\begin{equation}\label{eq:u2-quad-1}
\normnorm{f}_{U^2}^4 \ll \normnorm{\hat f}_\infty |H|^2p^{-6d}.
\end{equation}

\end{corollary}

\begin{proof}
By Lemma \ref{lem:struct-diff}, for fixed $y_0 \in \Psi_h$ we have
$
\Psi_h = Q^{-1}(0)\cap (V_h+y_0)
$, where $V_h$ denotes the subspace $\set{x \in H: B(h,x) + Lx = 0} = \ang{h^T B + L}^\perp$. Lemma \ref{lem:pseud-quad-diff} then gives that
\[
\biggabs{\sum_{x \in\Psi_h} e_p(\ell x) -
1_{V_h^\perp}(\ell)e_p(\ell y_0) |V_h|p^{-d}} \leq |H|p^{-R/2} . 
\]
The inequality \eqref{eq:pseud-once-diff-minus} now follows on combining the observation that  $V_h^\perp = \ang{h^TB+L}$ with Lemma \ref{lem:size-of-diff-sub} (which tells us that $\codim V_h = d$ for all but $ |H|p^{d - R}$ values of $h$).

The same argument gives 
\begin{equation}\label{eq:pseud-once-diff-no-eps}
\biggabs{\sum_{x \in\Phi_h} e_p(\ell x) -1_{\ang{Bh}}(\ell)
e_p(\ell x_0) |\ang{Bh}^\perp|p^{-d}} \leq |H|p^{-R/2} ,
\end{equation}
yielding \eqref{eq:pseud-once-diff} in conjunction with Lemma \ref{lem:size-of-diff-sub}. 

The inequality \eqref{eq:pseud-once-diff-no-eps} also tells us that either $R \ll d$ or $\Phi_h-x_0$ is sufficiently pseudorandom on $V_h$ for Theorem \ref{thm:restriction-inverse} to be applicable with $q = 4$, giving \eqref{eq:u2-quad-1}.
\end{proof}

\begin{corollary}[Pseudorandomness of thrice differenced quadratic level sets]\label{cor:pseud-thrice-diff}
Let $H$ be a finite vector space over $\F_p$ with $p$ odd and let $Q = (q_1, \dots, q_d)$ be a $d$-tuple of quadratic polynomials on $H$ of rank at least $R$ with associated $d$-tuple  of symmetric bilinear forms $B$. Writing
$$
V_{h_1, h_2, h_3} := \set{ x \in H : B(h_1, x) = B(h_2, x) = B(h_3, x) = 0},
$$
set 
$$
\Phi_{h_1, h_2, h_3}(x_0) := Q^{-1}(0)\cap (V_{h_1, h_2, h_3}+x_0)
$$
where $x_0 \in H$ is fixed. For all but $ |H|^3p^{3d - R}$ triples of $(h_1, h_2, h_3) \in H^3$ we have 
$$
\Bigabs{\sum_{x \in \Phi_{h_1, h_2, h_3}(x_0)} e_p(\ell(x-x_0)) - 1_{\ell \in \ang{h_1^TB, h_2^TB, h_3^TB}}|H|p^{-4d}} \leq  |H|p^{R/2}.
$$
\end{corollary}

\begin{proof}
Since $V_{h_1, h_2, h_3}$ is a subspace of codimension at most $3d$, Lemma \ref{lem:pseud-quad-diff} gives that
\[
\biggabs{\sum_{x \in\Phi_{h_1, h_2, h_3}(x_0)} e_p(\ell x) -1_{V_{h_1, h_2, h_3}^\perp}(\ell)
e_p(\ell x_0) |V_{h_1, h_2, h_3}|p^{-d}} \leq |H|p^{-R/2} . 
\]
The result follows on combining the observation that  $V_{h_1, h_2, h_3}^\perp = \ang{h_1^TB, h_2^TB, h_3^TB}$ with Lemma \ref{lem:size-of-diff-sub} (which tells us that $\dim\ang{h_1^TB, h_2^TB, h_3^TB} = 3d$ for all but $ |H|^3p^{3d - R}$ triples $(h_1, h_2, h_3)$.
\end{proof}

\section{Many partial additive quadruples}\label{sec:many-partial}
This section begins our proof of the relative $U^3$-inverse theorem (Theorem \ref{thm:u3-inverse}). We first show that if the $U^3$-norm of $f$ is large, then the difference functions $\Delta_h f$ have a large Fourier coefficient for a dense set of $h \in H$. This is an application of our $U^2$-inverse theorem relative to pseudorandom sets (Theorem \ref{thm:u2-inverse}).
\begin{lemma}[Many large Fourier coefficients]\label{lem:large-fourier}
Let $H$ be a finite vector space over $\F_p$ with $p$ odd and let $Q = (q_1, \dots, q_d)$ be a $d$-tuple of quadratic polynomials on $H$ whose rank (as a $d$-tuple) is at least $R$. For any 1-bounded function $f : H \to \C$ with $\supp(f) \subset Q^{-1}(0)$, if $\norm{f}_{U^3} \geq \eta \norm{1_{Q^{-1}(0)}}_{U^3}$, then  either $R\ll d + \log(2/\eta)$ or there exists a set $\mathcal{H} \subset H$ with $|\mathcal{H}| \gg \eta^{O(1)}|H|$ and a function $\phi : \mathcal{H} \to H^*$ such that for each $h \in \mathcal{H}$ we have
\begin{equation}\label{eq:large-fourier}
\bigabs{\widehat{\Delta_h f}(\phi(h))} \gg \eta^{O(1)}|H|p^{-2d}.
\end{equation}
\end{lemma}
\begin{proof}
We note that $\Delta_h f$ is a function supported on $\Phi_h := Q^{-1}(0)\cap[Q^{-1}(0) - h]$. By definition of the $U^3$-norm (Definition \ref{def:u3}), we have 
$$
\sum_h \norm{\Delta_hf}_{U^2}^4 \geq \eta^8 \sum_h \norm{1_{\Phi_h}}_{U^2}^4.
$$
By considering only the trivial Fourier coefficient we have
$$
\norm{1_{\Phi_h}}_{U^2}^4 = \E_\gamma \bigabs{\hat{1}_{\Phi_h}(\gamma)}^4 \geq |\Phi_h|^4 |H|^{-1}
$$
Using Corollary \ref{cor:pseud-once-diff}, either $R \ll d + \log(2/\eta)$ or for all but $\trecip{2}\eta^8 |H|$ values of $h$ we have $|\Phi_h| \asymp  |H|p^{-2d}$ and $\norm{\Delta_hf}_{U^2}^4 \ll \normnorm{\widehat{\Delta_h f}}_\infty |H|^2p^{-6d}$. Hence
\[
\sum_h\normnorm{\widehat{\Delta_h f}}_\infty \gg \eta^8 |H|^2p^{-2d}.\qedhere
\]
\end{proof}

Next we show that if \eqref{eq:large-fourier} holds, then the frequencies $\phi(h)\in H^*$ appearing in Lemma \ref{lem:large-fourier} have some additive structure. In the usual $U^3$-inverse argument of Gowers \cite{GowersNewFour} and Green and Tao \cite{GreenTaoInverse}, this structure is the assertion that for many solutions to the equation $h_1 + h_2 = h_3 + h_4$ we also have $\phi(h_1) +\phi(h_2) = \phi(h_3) + \phi(h_4)$, so that $\phi$ respects many additive quadruples. Adapting Gowers' argument to high-rank quadratic level sets does not give this conclusion, but gives a weaker additive relation we term many partial additive quadruples (described in the following lemma). This weaker notion is in some sense necessary given only the information \eqref{eq:large-fourier}. However in \S\ref{sec:many-genuine} we show that we can modify $\phi$ to give a new function where \eqref{eq:large-fourier} still holds, and where we do have many genuine additive quadruples.

\begin{lemma}[Many partial additive quadruples]\label{lem:many-partial-quads}
Let $H$ be a finite vector space over $\F_p$ with $p$ odd and let $Q = (q_1, \dots, q_d)$ be a $d$-tuple of quadratic polynomials on $H$ whose rank (as a $d$-tuple) is at least $R$. Write $B$ for the $d$-tuple of symmetric bilinear forms associated to $Q$. Let $\phi : \mathcal{H} \to H^*$ be a function, where $\mathcal{H} \subset H$ satisfies $|\mathcal{H}|\geq \eta |H|$. Suppose that there exists a 1-bounded function $f : H \to \C$ with $\supp(f) \subset Q^{-1}(0)$ such that for each $h \in \mathcal{H}$ we have 
\begin{equation*}
\bigabs{\widehat{\Delta_h f}(\phi(h))} \geq \eta |H|p^{-2d}.
\end{equation*}
Then either $R\ll d + \log(2/\eta)$ or there are $\gg \eta^{O(1)}|H|^3$ quadruples $(h_1, h_2, h_3, h_4) \in \mathcal{H}^4$ satisfying $h_1-h_2=h_3-h_4$ and
\begin{align*}
\phi(h_1) - \phi(h_2) -\phi(h_3)+\phi(h_4) \in  \ang{h_1^TB}-\ang{h_2^TB}-\ang{h_3^TB} + \ang{h_4^TB}.
\end{align*}
We recall (see \S\ref{sec:notation}) that $\ang{h^TB}$ denotes the $\F_p$-span of the linear forms $\ang{h^Tb_1, \dots, h^Tb_d}$.
\end{lemma}
\begin{proof}
Summing over $\mathcal{H}$ and opening absolute values by introducing a modulation, there exists a 1-bounded function $g : H \to \C$ with $\supp(g) \subset \mathcal{H}$  such that
$$
\eta^2 |H|^2p^{-2d} \leq \sum_{x }f(x) \sum_h g(h)\overline{f(x+h)}e_p(\phi(h)x).
$$
By Corollary \ref{cor:quad-level-size}, we may assume that $\sum_x |f(x)|^2 \ll |H|p^{-d}$, so applying Cauchy-Schwarz and changing variables gives
\begin{multline*}
\eta^{O(1)} |H|^3p^{-3d} \ll \sum_{x,h_1, h_2 }1_{Q^{-1}(0)}(x)  g(h_1)\overline{g(h_2)}\overline{f(x+h_1)}f(x+h_2)e_p\brac{\sqbrac{\phi(h_1)-\phi(h_2)}x}\\
	= \sum_{x,y_1, y_2 }1_{Q^{-1}(0)}(x)  g(y_1-x)\overline{g(y_2-x)}\overline{f(y_1)}f(y_2)e_p\brac{\sqbrac{\phi(y_1-x)-\phi(y_2-x)}x}.
\end{multline*}
Applying Cauchy-Schwarz again to remove the remaining $f$ weights, we have
\begin{align*}
\eta^{O(1)} |H|^4p^{-4d} \ll &\sum_{x_1,x_2,y_1, y_2 } 1_{Q^{-1}(0)}(x_1)1_{Q^{-1}(0)}(x_2)1_{Q^{-1}(0)}(y_1)1_{Q^{-1}(0)}(y_2) \\ &\qquad \times g(y_1-x_1)\overline{g(y_2-x_1)}\overline{g(y_1-x_2)}g(y_2-x_2)\\  &\qquad \times e_p\brac{\sqbrac{\phi(y_1-x_1)-\phi(y_2-x_1)}x_1-\sqbrac{\phi(y_1-x_2)-\phi(y_2-x_2)}x_2}.
\end{align*}
We reparametrise, setting $y := y_1-x_1$, $y+h_1:=y_2-x_1$, $y+h_2:=y_1-x_2$ (so that $y+h_1+h_2= y_2-x_2$) and $x := x_1$, then one can check that the latter is equal to
\begin{equation}\label{eq:many-add-sum}
\begin{split}
 \sum_{x,y,h_1,h_2 } 1_{Q^{-1}(0)}(x)1_{Q^{-1}(0)}&(x-h_2)1_{Q^{-1}(0)}(x+y)1_{Q^{-1}(0)}(x+y+h_1) \\  \times g(y)\overline{g(y+h_1)} &\overline{g(y+h_2)}g(y+h_1+h_2)\\  \times &e_p\brac{\sqbrac{\phi(y)-\phi(y+h_1)}x-\sqbrac{\phi(y+h_2)-\phi(y+h_1+h_2)}(x-h_2)}.
\end{split}
\end{equation}

Write
$$
\Psi_{y,h_1,h_2} := Q^{-1}(0) \cap \sqbrac{Q^{-1}(0) + h_2}\cap \sqbrac{Q^{-1}(0)-y}\cap \sqbrac{Q^{-1}(0)-y-h_1}.
$$
On employing Lemma \ref{lem:struct-diff} we see that, in the notation of Corollary \ref{cor:pseud-thrice-diff}, for any $x_0 \in \Psi_{y,h_1,h_2}$ we have $\Psi_{y, h_1, h_2} = \Phi_{y, y+h_1, -h_2}(x_0)$, and by the triangle inequality we have
\begin{multline*}
\eta^{O(1)} |H|^4p^{-4d} \ll  \sum_{y,h_1,h_2} 1_\mathcal{H}(y)1_\mathcal{H}(y+h_1)1_\mathcal{H}(y+h_2)1_\mathcal{H}(y+h_1+h_2)\times \\\Bigabs{\sum_x 1_{\Psi_{y,h_1,h_2}}(x) e_p\brac{\sqbrac{\phi(y)-\phi(y+h_1)-\phi(y+h_2)+\phi(y+h_1+h_2)}x}}
\end{multline*}
Employing Corollary \ref{cor:pseud-thrice-diff} to estimate the inner sum, we deduce that either $R \ll d + \log(2/\eta)$ or 
\begin{multline*}
\eta^{O(1)} |H|^3 \ll  \sum_{y,h_1,h_2} 1_\mathcal{H}(y)1_\mathcal{H}(y+h_1)1_\mathcal{H}(y+h_2)1_\mathcal{H}(y+h_1+h_2)\times \\1_{\phi(y)-\phi(y+h_1)-\phi(y+h_2)+\phi(y+h_1+h_2) \in \ang{y^TB, h_1^TB, h_2^TB}}.\qedhere
\end{multline*}
\end{proof}

\section{Many genuine additive quadruples}\label{sec:many-genuine}
The purpose of this section is to prove the following result, obtained independently by Gowers, Green and Tao \cite{gowers-additive} (preceding our work by almost two decades).
\begin{lemma}[Many partial additive quadruples give many genuine additive quadruples]\label{lem:partial-to-genuine}
Let $H$ be a finite vector space over $\F_p$ and let $B = (b_1, \dots, b_d)$ be a $d$-tuple of symmetric bilinear forms on $H$ whose rank (as a $d$-tuple) is at least $R$. Suppose that there exists a set $\mathcal{H} \subset H$ and a function $\phi : \mathcal{H} \to H^*$  such that there are at least $\delta |H|^3$ triples $(y_1,y_2, h)$ satisfying  $y_1, y_2, y_1+h, y_2+h \in \mathcal{H}$ and\footnote{Since the $b_i$ are symmetric, we recall (see \S\ref{sec:notation}) that $y^TB$ and $By$ give the same $d$-tuples in $(H^*)^d$.}
\begin{equation*}\label{eq:partial-quad}
\phi(y_1) - \phi(y_1+h) - \phi(y_2) + \phi(y_2+h) \in \ang{B y_1,By_2, Bh}.
\end{equation*}
Then either $R\ll d + \log(2/\delta)$ or  there exist coefficients $\mu(y)\in \F_p^d$ such that, on setting 
$$
\phi'(y):= \phi(y) + \mu(y)By ,
$$ there are at least $\gg \delta^{O(1)} |H|^3$ triples $(y_1,y_2, h)$ satisfying  $y_1, y_2, y_1+h, y_2+h \in \mathcal{H}$ and
\begin{equation*}
\phi'(y_1) -\phi'(y_1+h) - \phi'(y_2) + \phi'(y_2+h) =0.
\end{equation*}
\end{lemma}

The proof of this result relies on three applications of the same idea, each using Cauchy-Schwarz and the linear independence of the coordinates of $By_1, By_2, By_3, By_4, By_5$ for a generic quintuple $(y_1, y_2, y_3, y_4, y_5)$ (see Lemma \ref{lem:size-of-diff-sub}). We separate the proof into three lemmas in order to ease notation and clarify ideas.

Throughout this section we write $By$ for the $d$-tuple $(b_1y, \dots, b_dy)$, which (by symmetry) coincides with the $d$-tuple of linear maps given by $(y^Tb_1, \dots, y^Tb_d)$. For $\lambda \in \F_p^d$ write $\lambda By$ for $\lambda_1b_1y+ \dots + \lambda_d b_dy$. We write $\ang{By}$ for the span $\ang{b_1y, \dots, b_dy}= \set{\lambda By : \lambda \in \F_p^d} \leq H^*$.

\begin{lemma}[Removing dependence on one out of three variables]\label{lem:dep-2to3}
Let $H$ be a finite vector space over $\F_p$ and let $B = (b_1, \dots, b_d)$ be symmetric bilinear forms on $H$ whose rank (as a $d$-tuple) is at least $R$. Suppose that there exists a set $\mathcal{H} \subset H$, function $\phi : \mathcal{H} \to H^*$ and coefficients $\lambda_1(y_1, y_2, h)$, $\lambda_2(y_1, y_2, h)$, $\mu(y_1,y_2, h) \in \F_p^d$ such that there are at least $\delta |H|^3$ triples $(y_1,y_2, h)$ satisfying  $y_1, y_2, y_1+h, y_2+h \in \mathcal{H}$ and
\begin{multline*}
\phi(y_1) - \phi(y_1+h) - \phi(y_2) + \phi(y_2+h) = \\ \lambda_1(y_1, y_2,h)B y_1 - \lambda_2(y_1, y_2,h) By_2 + \mu(y_1, y_2, h) Bh.
\end{multline*}
Then either $R\ll d + \log(2/\delta)$ or there exist coefficients $\lambda(y,h)$, $\mu(y,h)\in \F_p^d$ such that there are at least $\trecip{2}\delta^4 |H|^3$ triples $(y_1,y_2, h)$ satisfying  $y_1, y_2, y_1+h, y_2+h \in \mathcal{H}$ and
\begin{multline*}
\phi(y_1) - \phi(y_1+h) - \phi(y_2) + \phi(y_2+h) \\ = \lambda(y_1,h)B (y_1+h) - \lambda(y_2, h) B(y_2+h) - \sqbrac{\mu(y_1,h)-\mu(y_2,h)}Bh.
\end{multline*}
\end{lemma}
\begin{proof}
It follows from the Cauchy-Schwarz inequality that at least $\delta^2 |H|^4$ quadruples $(y_1, y_2, y_2', h)$ satisfy $y_2,y_2', y_2+h, y_2'+h \in \mathcal{H}$ and
\begin{multline}\label{eq:doubled-y2}
\phi(y_2') - \phi(y_2'+h) - \phi(y_2 ) + \phi(y_2+h) 
 = \lambda_1(y_1, y_2, h) By_1- \lambda_2(y_1, y_2, h) By_2 + \mu(y_1, y_2, h) Bh\\ - \sqbrac{\lambda_1(y_1, y_2', h) By_1- \lambda_2(y_1, y_2', h) By_2' + \mu(y_1, y_2', h) Bh}\\
= \lambda_2(y_1, y_2', h) By_2'- \lambda_2(y_1, y_2, h) By_2 + \sqbrac{\mu(y_1, y_2, h)-\mu(y_1, y_2', h)} Bh\\  + \sqbrac{\lambda_1(y_1, y_2, h)-\lambda_1(y_1, y_2', h)} By_1.
\end{multline}
Re-applying the Cauchy-Schwarz inequality to double the $y_1$ variable, for at least $\delta^4 |H|^5$ quintuples $(y_1, y_1', y_2, y_2', h)$ we have $y_2,y_2', y_2+h, y_2'+h \in \mathcal{H}$ along with the identity \eqref{eq:doubled-y2}  and 
\begin{multline*}
\lambda_2(y_1, y_2', h) By_2'- \lambda_2(y_1, y_2, h) By_2 + \sqbrac{\mu(y_1, y_2, h)-\mu(y_1, y_2', h)} Bh\\  + \sqbrac{\lambda_1(y_1, y_2, h)-\lambda_1(y_1, y_2', h)} By_1\\
= \lambda_2(y_1', y_2', h) By_2'- \lambda_2(y_1', y_2, h) By_2 + \sqbrac{\mu(y_1', y_2, h)-\mu(y_1', y_2', h)} Bh\\  + \sqbrac{\lambda_1(y_1', y_2, h)-\lambda_1(y_1', y_2', h)} By_1'.
\end{multline*}
In particular, we have the linear dependence
\begin{equation}\label{eq:y2'-lin-dep}
\sqbrac{\lambda_1(y_1, y_2, h)-\lambda_1(y_1, y_2', h)} By_1 \in \ang{By_1', By_2, By_2', Bh}.
\end{equation}

By Lemma \ref{lem:size-of-diff-sub}, for all but $|H|^5p^{5d-R}$ choices of $(y_1, y_1', y_2, y_2', h)$, the $5d$ linear forms which constitute the coordinates of $By_1$, $By_1'$, $By_2$, $By_2'$, $Bh$ are linearly independent. For such choices, the linear dependence \eqref{eq:y2'-lin-dep} can only hold if $\lambda_1(y_1, y_2, h)=\lambda_1(y_1, y_2', h)$. Hence either $R \ll d + \log(2/\delta)$ or there exists $y_1$ such that for at least $\trecip{2}\delta^4 |H|^3$ triples $( y_2, y_2', h)$ we have $y_2,y_2', y_2+h, y_2'+h \in \mathcal{H}$ and
\begin{multline*}
\phi(y_2') - \phi(y_2'+h) - \phi(y_2 ) + \phi(y_2+h) \\
= \lambda_2(y_1, y_2', h) By_2'- \lambda_2(y_1, y_2, h) By_2 
+ \sqbrac{\mu(y_1, y_2, h)-\mu(y_1, y_2', h)} Bh
\\
= \lambda_2(y_1, y_2', h)B(y_2'+h)- \lambda_2(y_1, y_2, h)B(y_2+h) \\
+ \sqbrac{\mu(y_1, y_2, h)-\lambda_2(y_1,y_2,h)-\mu(y_1, y_2', h)+\lambda_2(y_1, y_2', h)} Bh.\qedhere
\end{multline*}
\end{proof}

\begin{lemma}[Removing dependence on two out of three variables]\label{lem:dep-1to3}
Let $H$ be a finite vector space over $\F_p$ and let $B = (b_1, \dots, b_d)$ be symmetric bilinear forms on $H$ whose rank (as a $d$-tuple) is at least $R$. Suppose that there exists a set $\mathcal{H} \subset H$, function $\phi : \mathcal{H} \to H^*$ and coefficients $\lambda(y,h)$, $\mu(y,h)\in \F_p^d$ such that there are at least $\delta |H|^3$ triples $(y_1,y_2, h)$ satisfying  $y_1, y_2, y_1+h, y_2+h \in \mathcal{H}$ and
\begin{multline*}
\phi(y_1) - \phi(y_1+h) - \phi(y_2) + \phi(y_2+h) \\ = \lambda(y_1,h)B (y_1+h) - \lambda(y_2, h) B(y_2+h) - \sqbrac{\mu(y_1,h)-\mu(y_2,h)}Bh.
\end{multline*}
Then either $R\ll d + \log(2/\delta)$ or there exist coefficients $\lambda' (y,h), \mu(y) \in \F_p^d$ such that, on setting $\phi'(y) := \phi(y)+\mu(y)By$, there are at least $\trecip{2}\delta^4 |H|^3$ triples $(y_1,y_2, h)$ satisfying  $y_1, y_2, y_1+h, y_2+h \in \mathcal{H}$ and
\begin{equation*}
\phi'(y_1) - \phi'(y_1+h) - \phi'(y_2) + \phi'(y_2+h) =\\  \lambda'(y_1,h) B(y_1+h) - \lambda'(y_2,h)B (y_2+h).
\end{equation*}
\end{lemma}

\begin{proof}
Write $\mu(y_1,y_2, h) := \mu(y_1,h)-\mu(y_2,h)$. By the Cauchy-Schwarz inequality, at least $\delta^2 |H|^4$ quadruples $(y_1, y_2, h,h')$ satisfy $y_1+h, y_2+h, y_2+h', y_2+h' \in \mathcal{H}$ and
\begin{multline*}
\phi(y_1+h') - \phi(y_1+h) - \phi(y_2 +h') + \phi(y_2+h) 
 = \\
 \lambda(y_1,h)B (y_1+h) - \lambda(y_2, h) B(y_2+h) - \mu(y_1,y_2, h)Bh\\ 
 - \lambda(y_1,h')B (y_1+h') + \lambda(y_2, h') B(y_2+h') + \mu(y_1,y_2, h')Bh'.
\end{multline*}

Re-parametrising variables, at least $\delta^2 |H|^4$ quadruples $(y_1, y_2, h,h')$ satisfy $y_1, y_1+h-h', y_2, y_2+h-h' \in \mathcal{H}$ and
\begin{multline*}
\phi(y_1) - \phi(y_1+h-h') - \phi(y_2 ) + \phi(y_2+h-h') 
 = \\
  \lambda(y_1-h',h)B (y_1+h-h') - \lambda(y_2-h', h) B(y_2+h-h') - \mu(y_1-h',y_2-h', h)Bh\\ 
 - \lambda(y_1-h',h')B y_1 +\lambda(y_2-h', h') By_2 + \mu(y_1-h',y_2-h', h')Bh'.
\end{multline*}
Re-parametrising variables once more, at least $\delta^2 |H|^4$ quadruples $(y_1, y_2, h,h')$ satisfy $y_1, y_2+h, y_2, y_2+h \in \mathcal{H}$ and
\begin{multline}\label{eq:doubled-h-0}
\phi(y_1) - \phi(y_1+h) - \phi(y_2 ) + \phi(y_2+h) 
 = \\
 -\lambda(y_1-h',h')B y_1+\lambda(y_1-h',h+h')B (y_1+h)  + \lambda(y_2-h', h') By_2- \lambda(y_2-h', h+h') B(y_2+h)\\ - \mu(y_1-h',y_2-h', h+h')Bh
  + \sqbrac{\mu(y_1-h',y_2-h', h')-\mu(y_1-h',y_2-h', h+h')}Bh'.
\end{multline}

Proceeding as in the proof of Lemma \ref{lem:dep-2to3}, we again apply  Cauchy-Schwarz to double the $h'$ variable, and conclude that for at least $\delta^4 |H|^5$ quintuples $(y_1,y_2, h, h', h'')$ we have $y_1, y_2+h, y_2, y_2+h \in \mathcal{H}$ along with the identity \eqref{eq:doubled-h-0}  and the linear dependence
\begin{equation}\label{eq:h'-lin-dep}
\sqbrac{\mu(y_1-h',y_2-h', h')-\mu(y_1-h',y_2-h', h+h')}Bh'\\ \in \ang{By_1, By_2, Bh, Bh''}.
\end{equation}

By Lemma \ref{lem:size-of-diff-sub}, for all but $O\brac{|H|^5p^{5d-R}}$ choices of $(y_1,  y_2, h,h',h'')$, the $5d$ linear forms which constitute the coordinates of $By_1$, $By_2$, $Bh$, $Bh'$, $Bh''$ are linearly independent. For such choices, the linear dependence \eqref{eq:h'-lin-dep} can only hold if $\mu(y_1-h',y_2-h', h')=\mu(y_1-h',y_2-h', h+h')$. Hence either $R \ll d + \log(2/\delta)$ or there exists $h'$ such that for at least $\trecip{2}\delta^4 |H|^3$ triples $( y_1, y_2, h)$ we have $y_1,y_2, y_1+h, y_2+h \in \mathcal{H}$ and
\begin{multline}\label{eq:doubled-h-1}
\phi(y_1) - \phi(y_1+h) - \phi(y_2 ) + \phi(y_2+h) 
 = \\
 -\lambda(y_1-h',h')B y_1+\lambda(y_1-h',h+h')B (y_1+h)  + \lambda(y_2-h', h') By_2- \lambda(y_2-h', h+h') B(y_2+h)\\ - \mu(y_1-h',y_2-h', h')Bh.
\end{multline}
Recalling that $\mu(y_1,y_2, h) := \mu(y_1,h)-\mu(y_2,h)$, we can use the identity 
$$
-\sqbrac{ \mu(y_1)-\mu(y_2)}Bh = \mu(y_1)By_1 - \mu(y_1)B(y_1+h)-\mu(y_2)By_2+\mu(y_2)B(y_2+h)
$$
to re-write \eqref{eq:doubled-h-1} in the form
\begin{multline*}
\phi(y_1) - \phi(y_1+h) - \phi(y_2 ) + \phi(y_2+h) 
 = \\
 -\mu(y_1)B y_1+\lambda'(y_1,h)B (y_1+h)  + \mu(y_2) By_2- \lambda'(y_2, h) B(y_2+h)
\end{multline*}
where $\mu(y) := \lambda(y-h',h') -\mu(y-h',h')$ and $\lambda'(y,h) = \lambda(y-h', h+h')-\mu(y-h', h')$. Hence on defining $\phi'(y) := \phi(y) + \mu(y)By$, we see that for at least $\trecip{2}\delta^4 |H|^3$ triples $( y_1, y_2, h)$ we have $y_1,y_2, y_1+h, y_2+h \in \mathcal{H}$ and
\begin{equation*}
\phi'(y_1) - \phi'(y_1+h) - \phi'(y_2 ) + \phi'(y_2+h) 
 = \\
 \lambda''(y_1,h)B (y_1+h) - \lambda''(y_2, h) B(y_2+h),
\end{equation*}
where $\lambda''(y,h) = \lambda'(y,h)-\mu(y+h)$.
\end{proof}

\begin{lemma}[Removing dependence on all three variables]\label{lem:dep-0to3}
Let $H$ be a finite vector space over $\F_p$ and let $B = (b_1, \dots, b_d)$ be symmetric bilinear forms on $H$ whose rank (as a $d$-tuple) is at least $R$. Suppose that there exists a set $\mathcal{H} \subset H$, function $\phi : \mathcal{H} \to H^*$ and coefficients $\lambda(y,h)\in \F_p^d$ such that there are at least $\delta |H|^3$ triples $(y_1,y_2, h)$ satisfying  $y_1, y_2, y_1+h, y_2+h \in \mathcal{H}$ and
\begin{equation*}
\phi(y_1) - \phi(y_1+h) - \phi(y_2) + \phi(y_2+h) =\\ \lambda(y_1,h)B (y_1+h) - \lambda(y_2, h) B(y_2+h).
\end{equation*}
Then either $R\ll d + \log(2/\delta)$ or there are at least $\trecip{4}\delta^8 |H|^3$ triples $(y_1,y_2, h)$ satisfying  $y_1, y_2, y_1+h, y_2+h \in \mathcal{H}$ and
\begin{equation*}
\phi(y_1) - \phi(y_1+h) - \phi(y_2) + \phi(y_2+h) =0.
\end{equation*}
\end{lemma}

\begin{proof}
Re-parametrising, there are at least $\delta |H|^3$ triples $(y_1,y_2, h)$ satisfying  $y_1, y_2, y_1+h, y_2-h \in \mathcal{H}$ and
\begin{equation*}
\phi(y_1) - \phi(y_1+h) - \phi(y_2-h) + \phi(y_2) =\\ \lambda(y_1,h)B (y_1+h) - \lambda(y_2-h, h) By_2.
\end{equation*}

By the Cauchy-Schwarz inequality, at least $\delta^2 |H|^4$ quadruples $(y_1, y_2, h,h')$ satisfy $y_1+h, y_2-h,y_1+h', y_2-h' \in \mathcal{H}$ and
\begin{multline*}
\phi(y_1+h') - \phi(y_1+h) - \phi(y_2-h) + \phi(y_2-h') =\\ \lambda(y_1,h)B (y_1+h)- \lambda(y_1,h')B (y_1+h') - \sqbrac{\lambda(y_2-h, h)-\lambda(y_2-h', h')} By_2.
\end{multline*}

Re-parametrising variables, at least $\delta^2 |H|^4$ quadruples $(y_1, y_2, h,h')$ satisfy $y_1,y_1+h-h', y_2, y_2+h-h' \in \mathcal{H}$ and
\begin{multline*}
\phi(y_1) - \phi(y_1+h-h') - \phi(y_2) + \phi(y_2+h-h') =\\ \lambda(y_1-h',h)B (y_1+h-h')- \lambda(y_1-h',h')B y_1 - \sqbrac{\lambda(y_2, h)-\lambda(y_2+h-h', h')} B(y_2+h).
\end{multline*}

Re-parametrising once more, at least $\delta^2 |H|^4$ quadruples $(y_1, y_2, h,h')$ satisfy $y_1, y_1+h, y_2, y_2+h \in \mathcal{H}$ and
\begin{multline}\label{eq:h'-dbl-2}
\phi(y_1) - \phi(y_1+h) - \phi(y_2) + \phi(y_2+h) \\ 
= \lambda(y_1-h',h+h')B (y_1+h)- \lambda(y_1-h',h')B y_1 - \sqbrac{\lambda(y_2, h+h')-\lambda(y_2+h, h')} B(y_2+h+h').
\end{multline}

Doubling the $h'$ variable via Cauchy-Schwarz, at least $\delta^4 |H|^5$ quintuples $(y_1,y_2, h, h', h'')$ satisfy $y_1, y_2+h, y_2, y_2+h \in \mathcal{H}$ along with the identity \eqref{eq:h'-dbl-2}  and the linear dependence
\begin{equation*}
\sqbrac{\lambda(y_2, h+h')-\lambda(y_2+h, h')}  Bh'\\ \in \ang{By_1, By_2, Bh, Bh''}.
\end{equation*}
Hence either $R \ll d + \log(2/\delta)$ or there exists $h'$ such that for at least $\trecip{2}\delta^4 |H|^3$ triples $( y_1, y_2, h)$ we have $y_1,y_2, y_1+h, y_2+h \in \mathcal{H}$ and
\begin{equation*}
\phi(y_1) - \phi(y_1+h) - \phi(y_2) + \phi(y_2+h) \\ 
= \lambda(y_1-h',h+h')B (y_1+h)- \lambda(y_1-h',h')B y_1.
\end{equation*}

Applying Cauchy-Schwarz to double the $y_2$ variable, for at least $\trecip{4}\delta^8 |H|^3$ triples $( y_2', y_2, h)$ we have $y_2',y_2, y_2'+h, y_2+h \in \mathcal{H}$ and
\begin{equation*}
\phi(y_2') - \phi(y_2'+h) - \phi(y_2) + \phi(y_2+h) \\ 
= 0.\qedhere
\end{equation*}
\end{proof}

Combining Lemmas \ref{lem:dep-2to3}, \ref{lem:dep-1to3} and \ref{lem:dep-0to3} gives a proof of Lemma \ref{lem:partial-to-genuine}. Combining the existence of many large Fourier coefficients (Lemma \ref{lem:large-fourier}), the fact that many large Fourier coefficients implies many partial additive quadruples (Lemma \ref{lem:many-partial-quads}) and the fact that partial quadruples can be upgraded to genuine quadruples (Lemma \ref{lem:partial-to-genuine}), we obtain the following.

\begin{corollary}[Many additive quadruples]\label{cor:many-add-quad}
Let $H$ be a finite vector space over $\F_p$ with $p$ odd and let $Q = (q_1, \dots, q_d)$ be a $d$-tuple of quadratic polynomials on $H$ whose rank (as a $d$-tuple) is at least $R$. 
Let $f : H \to \C$ be a 1-bounded function with $\supp(f) \subset Q^{-1}(0)$ and suppose that $\norm{f}_{U^3} \geq \eta \norm{1_{Q^{-1}(0)}}_{U^3}$. Then either $R\ll d + \log(2/\eta)$ or there exists a set $\mathcal{H} \subset H$ with $|\mathcal{H}|\gg \eta^{O(1)}|H|$ and a function $\phi : \mathcal{H} \to H^*$ such that for each $h \in \mathcal{H}$ we have
\begin{equation*}
\bigabs{\widehat{\Delta_h f}(\phi(h))} \gg \eta^{O(1)} |H|p^{-2d},
\end{equation*}
and there are $\gg \eta^{O(1)}|H|^3$ quadruples $(h_1, h_2, h_3, h_4) \in \mathcal{H}^4$ satisfying $h_1-h_2=h_3-h_4$ and
$\phi(h_1) - \phi(h_2) =\phi(h_3)-\phi(h_4)$.
\end{corollary}
\begin{proof}
Combining Lemmas \ref{lem:large-fourier}, \ref{lem:many-partial-quads} and \ref{lem:partial-to-genuine}, our hypotheses imply the existence of a set $\mathcal{H} \subset H$ with $|\mathcal{H}|\gg \eta^{O(1)}|H|$ and functions $\phi : \mathcal{H} \to H$ and $\mu : \mathcal{H}\to \F_p^d$ such that for each $h \in \mathcal{H}$ we have
\begin{equation*}
\bigabs{\widehat{\Delta_h f}(\phi(h))} \gg \eta^{O(1)} |H|p^{-2d},
\end{equation*}
and on setting $\phi'(h) = \phi(h) + \mu(h)By$ we have $\gg \eta^{O(1)}|H|^3$ quadruples $(h_1, h_2, h_3, h_4) \in \mathcal{H}^4$ satisfying $h_1-h_2=h_3-h_4$ and
$\phi'(h_1) - \phi'(h_2) =\phi'(h_3)-\phi'(h_4)$. 

We observe that the function $\Delta_h f$ is supported on $Q^{-1}(0)\cap \sqbrac{Q^{-1}(0) - h}$. By Lemma \ref{lem:struct-diff}, if we fix $x_h \in Q^{-1}(0)\cap \sqbrac{Q^{-1}(0) - h}$ then
$$
Q^{-1}(0)\cap \sqbrac{Q^{-1}(0) - h} - x_h = \sqbrac{Q^{-1}(0)-x_h}\cap V_h,
$$
where $V_h = \set{x\in H : B(h,x) = 0}$. Note that if $x \in V_h$ then $\phi(h)x = \phi'(h)x$, and so
\begin{multline*}
\bigabs{\widehat{\Delta_h f}(\phi(h))}  = \Bigabs{\sum_{x\in \sqbrac{Q^{-1}(0)-x_h}\cap V_h} \Delta_hf(x_h + x) e_p(-\phi(h)x)}\\ = \Bigabs{\sum_{x\in \sqbrac{Q^{-1}(0)-x_h}\cap V_h} \Delta_hf(x_h + x) e_p(-\phi'(h)x)}  = \bigabs{\widehat{\Delta_h f}(\phi'(h))}.\qedhere
\end{multline*}
\end{proof}
\section{The symmetry argument and proof of the relative inverse theorem}\label{sec:sym}
We begin this section by showing how the function $\phi$ appearing in the statement of Corollary \ref{cor:many-add-quad} can be upgraded to take the form of an affine linear map $\phi(h) = h^TM +\ell$. The argument originates with Gowers \cite{GowersNewFour} and was adapted to finite vector spaces by Green and Tao \cite{GreenTaoInverse}. However, the strong quantitative variant we employ is essentially equivalent to the polynomial Freiman-Ruzsa conjecture, whose recent proof \cite{gowers2024martons} (due to Gowers, Green, Manners and Tao) does the heavy lifting in our argument. 
\begin{lemma}[Large $U^3$ gives linear correlation in the derivative]\label{lem:lin-cor}
Let $H$ be a finite vector space over $\F_p$ with $p$ odd and let $Q = (q_1, \dots, q_d)$ be a $d$-tuple of quadratic polynomials on $H$ whose rank (as a $d$-tuple) is at least $R$. Let $f : H \to \C$ be a 1-bounded function with $\supp(f) \subset Q^{-1}(0)$ and suppose that $\norm{f}_{U^3} \geq \eta \norm{1_{Q^{-1}(0)}}_{U^3}$. Then either $R\ll d + \log(2/\eta)$ or there exists a linear map $h \mapsto h^TM$ from $H$ to $H^*$ and a constant $\ell \in H^*$ such that
\begin{equation}\label{eq:phi-linear}
\sum_h\bigabs{\widehat{\Delta_h f}(h^TM+\ell)} \gg_p \eta^{O_p(1)} |H|^2p^{-2d}.
\end{equation}
\end{lemma}
\begin{proof}
Employing Corollary \ref{cor:many-add-quad}, we may assume that there exists a set $\mathcal{H} \subset H$ with $|\mathcal{H}|\gg \eta^{O(1)}|H|$ and a function $\phi : \mathcal{H} \to H^*$ such that for each $h \in \mathcal{H}$ we have
\begin{equation*}
\bigabs{\widehat{\Delta_h f}(\phi(h))} \gg \eta^{O(1)} |H|p^{-2d},
\end{equation*}
and there are $\gg \eta^{O(1)}|H|^3$ quadruples $(h_1, h_2, h_3, h_4) \in \mathcal{H}^4$ satisfying $h_1-h_2=h_3-h_4$ and
$\phi(h_1) - \phi(h_2) =\phi(h_3)-\phi(h_4)$.

We now follow \cite[Appendix C]{gowers2024martons} (which in turn follows \cite[\S5]{GreenTaoInverse}) up to the final displayed equation preceding \cite[(C.1)]{gowers2024martons}, all of which remains valid in our context. Hence we may assume the existence of a linear map which takes $h\in H$ to $ h^TM \in H^*$ and a constant $\ell \in H^*$ such that  
\begin{equation*}
\sum_{h \in \mathcal{H}} 1_{\phi(h) = h^TM + \ell} \gg_p \eta^{O_p(1)} |H|.
\end{equation*}
The result follows.
\end{proof}

To upgrade \eqref{eq:phi-linear}, which gives linear correlation in the derivative, to global quadratic correlation, we wish to upgrade the bilinear map $(h,x) \mapsto h^TMx$ to one which is symmetric. Such a procedure was pioneered in Green and Tao's ``symmetry argument'' \cite[\S6, Step 2]{GreenTaoInverse}. One can adapt their procedure to work relative to high-rank quadratic level sets, with help from the pseudorandom spectral estimate given in Lemma \ref{lem:spec-estimate}. Instead, we take the opportunity to give an alternative version of the symmetry argument, whose adaptation to the relative context is a little simpler. We first given the non-relative variant of our symmetry argument.

\begin{proposition}[Symmetry argument for densely supported functions]
Let $H$ be a finite vector space over $\F_p$ with $p$ odd. Let $f : H \to \C$ be a 1-bounded function and suppose that there exists a linear map $h \mapsto h^TM$ from $H$ to $H^*$ and a constant $\ell\in H^*$ such that
\begin{equation}\label{eq:phi-linear-2}
\sum_h\bigabs{\widehat{\Delta_h f}(h^TM+\ell)} \geq \eta |H|^2.
\end{equation}
Then, on defining the bilinear map $M^T$ by $(x,y)\mapsto y^T M x$, we have
\begin{equation*}
\sum_h\bigabs{\widehat{\Delta_h f}(\trecip{2}h^T[M+M^T])} \geq \eta^4|H|^2.
\end{equation*}
\end{proposition}
\begin{proof}
We apply Cauchy-Schwarz to \eqref{eq:phi-linear-2} to give
\begin{equation*}
\eta^2 |H|^3 \leq \sum_{x_1, x_2, h} f(x_1)\overline{f(x_1+h)f(x_2)}f(x_2+h) e_p\sqbrac{h^TM(x_1-x_2) + \ell(x_1-x_2)}.
\end{equation*}
Re-parametrising, this is equivalent to 
\begin{equation*}
\eta^2 |H|^3 \leq \sum_{x_1, x_2, h} f(x_1)\overline{f(x_1+h)f(x_2-h)}f(x_2) e_p\sqbrac{h^TM(x_1-x_2+h) + \ell(x_1-x_2+h)}.
\end{equation*}
Applying Cauchy-Schwarz to double the $h$ variable, we obtain
\begin{multline*}
\eta^4 |H|^4 \leq \sum_{x_1, x_2, h, h'} \overline{f(x_1+h)f(x_2-h)}f(x_1+h')f(x_2-h')\\\times e_p\sqbrac{h^TM(x_1-x_2+h) - h'^TM(x_1-x_2+h') + \ell(h-h')}.
\end{multline*}
Re-parametrising, first by changing variables $(x_1,x_2) \rightarrow (x_1-h',x_2+h)$, then by a further change of variables $h \rightarrow h+h'$, we obtain 
\begin{multline*}
\eta^4 |H|^4 \leq  \sum_{x_1, x_2, h}f(x_1) \overline{f(x_1+h)f(x_2)}f(x_2+h)e_p\sqbrac{h^TM(x_1-x_2)+\ell h}\\\times \sum_{h'}e_p\sqbrac{h^T(M-M^T)h' }\\
= |H|\sum_{\substack{h\\ h^T(M-M^T) = 0}}e_p(\ell h) \sum_{x_1, x_2}f(x_1) \overline{f(x_1+h)f(x_2)}f(x_2+h)e_p\sqbrac{h^TM(x_1-x_2)}\\
\leq |H| \sum_{\substack{h\\ h^T(M-M^T) = 0}}\bigabs{\widehat{\Delta_h f}(h^TM)}^2
\leq |H|^2\sum_{h} \bigabs{\widehat{\Delta_h f}\sqbrac{h^T(M+M^T)/2}}.\qedhere
\end{multline*}
\end{proof}

\begin{lemma}[Symmetry argument relative to quadratic level sets]\label{lem:sym-arg}
Let $H$ be a finite vector space over $\F_p$ with $p$ odd and let $Q = (q_1, \dots, q_d)$ be a $d$-tuple of quadratic polynomials on $H$ whose rank (as a $d$-tuple) is at least $R$. Let $0< \eta \leq 1$ and let $f : H \to \C$ be a 1-bounded function with $\supp(f) \subset Q^{-1}(0)$. Suppose that there exists a linear map $h \mapsto h^TM$ from $H$ to $H^*$ and a constant $\ell\in H^*$ such that
\begin{equation}\label{eq:phi-linear-3}
\sum_h\bigabs{\widehat{\Delta_h f}(h^TM+\ell)} \geq \eta |H|^2p^{-2d}.
\end{equation}
Then either $R \ll d + \log(2/\eta)$ or 
\begin{equation*}\label{eq:phi-linear-4}
\sum_h\bigabs{\widehat{\Delta_h f}(\trecip{2}h^T[M+M^T])} \gg \eta^{O(1)}|H|^2p^{-2d}.
\end{equation*}
\end{lemma}

\begin{proof}
We apply Cauchy-Schwarz to \eqref{eq:phi-linear-3} to give
\begin{equation*}
\eta^2 |H|^3p^{-4d} \leq \sum_{x_1, x_2, h} f(x_1)\overline{f(x_1+h)f(x_2)}f(x_2+h) e_p\sqbrac{h^TM^T(x_1-x_2) + \ell(x_1-x_2)}.
\end{equation*}
Re-parametrising, this is equivalent to 
\begin{multline*}
\eta^2 |H|^3p^{-4d} \leq\\ \sum_{x_1, x_2, h} f(x_1)\overline{f(x_1+h)f(x_2-h)}f(x_2) e_p\sqbrac{h^TM^T(x_1-x_2+h) +\ell (x_1-x_2+h)}.
\end{multline*}

Since the support of $f$ lies in $Q^{-1}(0)$, we can use Corollary \ref{cor:quad-level-size} to give the estimate
$
\sum_x |f(x)|^2 \ll |H|p^{-d},
$
or else $R \ll d$. Thus applying Cauchy-Schwarz to double the $h$ variable, we obtain
\begin{multline*}
\eta^4 |H|^4p^{-6d} \ll \sum_{x_1,x_2,h,h'} 1_{Q^{-1}(0)}(x_1) 1_{Q^{-1}(0)}(x_2)\overline{f(x_1+h)f(x_2-h)}f(x_1+h')f(x_2-h')\\\times e_p\sqbrac{h^TM(x_1-x_2+h) - h'^TM(x_1-x_2+h') + \ell(h-h')}.
\end{multline*}
Re-parametrising, first by changing variables $(x_1,x_2) \rightarrow (x_1-h',x_2+h)$, then by a further change of variables $h \rightarrow h+h'$, this is equivalent to
\begin{multline}\label{eq:orthog-inner-sum}
\eta^4 |H|^4 p^{-6d}\ll  \sum_{x_1, x_2, h}f(x_1) \overline{f(x_1+h)f(x_2)}f(x_2+h)e_p\sqbrac{h^TM(x_1-x_2)+ \ell h}\\\times \sum_{h'}1_{Q^{-1}(0)}(x_1-h') 1_{Q^{-1}(0)}(x_2+h+h')e_p\sqbrac{h^T(M^T-M)h' }\\
=\sum_{x,h,k}f(x) \overline{f(x+h)f(x+k)}f(x+h+k)e_p\brac{-h^TMk+\ell h}\\\times \sum_{y}1_{Q^{-1}(0)}(x-y) 1_{Q^{-1}(0)}(x+h+k+y)e_p\sqbrac{h^T(M^T-M)y }
\end{multline}

We would like to determine the set of $y$ contributing to the innermost sum; our aim is to show that this is a pseudorandom set, so the inner sum can only be large if $h^T(M^T - M) = 0$.  We first need to determine the set of $(x,h,k)$ contributing to the outermost sum. Writing $B$ for the $d$-tuple of symmetric bilinear forms associated to $Q$, one can check that
\begin{multline}\label{eq:quad-in-quad}
Q(x) = Q(x+h) = Q(x+k)  = Q(x+h+k) = 0;\\
 \iff\\
Q(x) =0,\quad
Q(h) + 2B(x,h) = Q(0),\quad
Q(k) + 2B(x,k)  =  Q(0), \quad
 B(h,k) = 0.
\end{multline}
Suppose that $(x,h,k)$ satisfy \eqref{eq:quad-in-quad}, so that this tuple contributes to the outer sum on the right-hand side of \eqref{eq:orthog-inner-sum}. Then, writing $L$ for the homogeneous linear part of $Q$, we have the following equivalence for the set of $y$ contributing to the inner sum on the right-hand side of \eqref{eq:orthog-inner-sum}:
\begin{multline*}
Q(x-y)  = Q(x+h+k+y) = 0
 \iff\\
 Q(y) = 2B(x,y)+2Ly +Q(0) \text{\quad\&\quad}
 B(2x+h+k, y) + Ly =0.
\end{multline*}

By orthogonality and the Weyl bound (Lemma \ref{lem:weyl-bound}), for any  subspace $V \leq H$  of codimension at most $d$, any linear map $\tilde L : V \to \F_p^d$,  any  $ \tilde \ell \in H^*$ and any constant $a \in \F_p^d$ we have
\begin{equation*}
\sum_{ \substack{y \in V\\Q(y) = \tilde Ly+a }}e_p\brac{\ell y } = \E_{\gamma\in \F_p^d}\sum_{y \in V} e_p\brac{\ell y + \gamma^T\sqbrac{Q(y) - \tilde L y-a}}
= |V|p^{-d}1_{\ell \in V^\perp}+ O(|V|p^{d-R/2}).
\end{equation*}
Therefore, returning to \eqref{eq:orthog-inner-sum}, either $R \ll d + \log(2/\eta)$ or on writing 
$$
V_{x,h,k} := \set{y\in H : B(2x+h+k, y) + Ly = 0} = \ang{(2x+h+k)^T B+L}^\perp,
$$ we have
\begin{multline*}
\eta^4 |H|^3 p^{-5d}\ll  \sum_{x, h,k}f(x) \overline{f(x+h)f(x+k)}f(x+h+k)e_p\brac{-h^TMk+ \ell h}\\\times  1_{h^T(M^T-M)\in V_{x,h,k}^\perp}\ p^{-\codim V_{x,h,k}}.
\end{multline*}

By Lemma \ref{lem:size-of-diff-sub}, for all but $|H|^3p^{d-R}$ triples $(x,h,k)$ we have
$
\codim V_{x,h,k} = d
$.
Consequently either $R \ll d + \log(2/\eta)$ or
\begin{multline*}
\eta^4 |H|^3 p^{-4d}\ll  \sum_{x, h,k}f(x) \overline{f(x+h)f(x+k)}f(x+h+k)\\ 
\times e_p\brac{-h^TMk+\ell h}  1_{h^T(M^T-M)\in V_{x,h,k}^\perp}.
\end{multline*}

Using Cauchy-Schwarz to double the $k$ variable gives
\begin{multline*}
\eta^8 |H|^4 p^{-6d}\ll \\ \sum_{x, h,k,k'}f(x+k')\overline{f(x+k)f(x+h+k')}f(x+h+k)1_{Q^{-1}(0)}(x)1_{Q^{-1}(0)}(x+h)\\ 
\times e_p\sqbrac{h^TM(k'-k)}  1_{h^T(M^T-M)\in V_{x,h,k}^\perp\cap V_{x,h,k'}^\perp}.
\end{multline*}
A change of variables of the form $x \mapsto x-k'$, followed by a further change of the form $k \mapsto k+k'$, we deduce that
\begin{multline*}
\eta^8 |H|^4 p^{-6d}\ll 
 \sum_{x, h,k}f(x)\overline{f(x+k)f(x+h)}f(x+h+k)\\ 
\times e_p\brac{-h^TMk}  \sum_{\substack{k'\\Q(x-k')=0\\Q(x+h-k') = 0}}1_{h^T(M^T-M)\in V_{x-k',h,k+k'}^\perp\cap V_{x-k',h,k'}^\perp}
\end{multline*}
As discussed in \S\ref{sec:notation},  we have
$$
V_{x,h,k}^\perp = \brac{\ang{B(2x+h+k) + L}^\perp}^\perp = \ang{B(2x+h+k) + L}.
$$
By Lemma \ref{lem:size-of-diff-sub},  the total number of quadruples $(x,h,k,k')$ for which 
$$
 \set{0} \neq V_{x-k',h,k+k'}^\perp\cap V_{x-k',h,k'}^\perp = \ang{B(2x+h+k-k')+L}\cap \ang{B(2x+h-k')+L} 
$$ 
is at most $|H|^4p^{2d-R}$. Therefore, either $R \ll d + \log(2/\eta)$ or
\begin{multline*}
\eta^8 |H|^4 p^{-6d}\ll 
 \sum_{x, h,k}f(x)\overline{f(x+k)f(x+h)}f(x+h+k)\\ 
\times e_p\brac{-h^TMk}1_{h^T(M^T-M) = 0}  \sum_{\substack{k'\\Q(x-k')=0\\Q(x+h-k') = 0}}1.
\end{multline*}
By Corollary \ref{cor:pseud-once-diff}, for fixed $x$ and $h$ we have 
$$
\sum_{\substack{k'\\Q(x-k')=0\\Q(x+h-k') = 0}}1 = |H|p^{-2d} + O(|H|p^{-R/2}).
$$
It follows that either $R \ll d + \log(2/\eta)$ or
\begin{multline*}
\eta^8 |H|^3 p^{-4d}\ll \sum_{\substack{x,h,k\\ h^T(M^T-M)=0}}f(x)\overline{f(x+k)f(x+h)}f(x+h+k)
 e_p\brac{-h^TMk}\\
 = \sum_{h^T(M^T-M)=0}\ \Bigabs{\sum_x f(x)\overline{f(x+h) } 
 e_p\sqbrac{\trecip{2}h^T(M+M^T)^Tx}}^2\\
 \leq \sum_{h}\bigabs{\widehat{\Delta_h f}(\trecip{2}h^T\sqbrac{M+M^T})}^2.
 \end{multline*}
 
 We can replace the $L^2$ bound above with the $L^1$ bound in \eqref{eq:phi-linear-4} by noting that $\Delta_h f$ is supported on $Q^{-1}(0)\cap \sqbrac{Q^{-1}(0)-h}$ then using Corollary \ref{cor:pseud-once-diff}. This gives
 \begin{equation*}
 \sum_{h}\bigabs{\widehat{\Delta_h f}(\trecip{2}h^T\sqbrac{M+M^T}h)}^2 \leq \brac{|H|p^{-2d} +O(|H|p^{-R/2})}\sum_h\bigabs{\widehat{\Delta_h f}(\trecip{2}\sqbrac{M^T+M}h)}.
\end{equation*}
Hence either $R \ll d +\log(2/\eta)$ or the result follows.
\end{proof}

\begin{lemma}[Integration argument]\label{lem:int-arg}
Let $H$ be a finite vector space over $\F_p$ with $p$ odd and let $Q = (q_1, \dots, q_d)$ be a $d$-tuple of quadratic polynomials on $H$ whose rank (as a $d$-tuple) is at least $R$. Let $f : H \to \C$ be a 1-bounded function with $\supp(f) \subset Q^{-1}(0)$. Suppose that there exists a symmetric bilinear form $(x,y) \mapsto x^TMy$  from $H\times H$ to $\F_p$  such that
\begin{equation*}
\sum_h\bigabs{\widehat{\Delta_h f}(h^TM)} \geq \eta|H|^2p^{-2d}.
\end{equation*}
Then either $R \ll d $ or there exists $\ell \in H^* $ such that
$$
\Bigabs{\sum_x f(x) e_p(-x^TMx + \ell x)}\gg \eta^{2} |Q^{-1}(0)| \asymp \eta^{2} |H|p^{-d}
$$
\end{lemma}

\begin{proof}
Introducing a modulation to remove absolute values, there exists a 1-bounded function $g : H \to \C$ such that
$$
\eta |H|^2p^{-2d} \leq \sum_{h,x}g(h)f(x)\overline{f(x+h)} e_p(h^TMx).
$$
The symmetry of $M$ gives  the identity $h^TMx = \trecip{2}\sqbrac{(x+h)^TM(x+h)-x^TMx -h^TMh }$. Hence
\begin{equation*}
\eta |H|^2p^{-2d} \leq \sum_{h,x}g(h)e_p(-h^TMh)f(x)e_p(-x^TMx)\overline{f(x+h)} e_p\sqbrac{(x+h)^TM(x+h)}.
\end{equation*}

Write $G(h) := g(h)e_p(-h^TMh)$ and $F(x) := f(x) e_p(-x^TMx)$. Then by orthogonality and Cauchy-Schwarz:
\begin{equation*}
\eta |H|^2p^{-2d} \leq \sum_{h,x}G(h)F(x)\overline{F(x+h)} = \E_{ \ell \in H^*} \hat{G}(\ell)\bigabs{\hat{F}(\ell)}^2 \leq \norm{G}_2 \norm{F}_{U^2}^2 \leq |H|^{1/2} \norm{F}_{U^2}^2.
\end{equation*}

Employing the $U^2$-inverse on high-rank quadratic level sets, as stated in \eqref{eq:u2-quad-0} of Corollary \ref{cor:quad-level-size}, either $R \ll d $ or  
$
\normnorm{\hat F}_\infty \gg \eta^2|H|p^{-d}
$.
\end{proof}

Combining Lemmas \ref{lem:lin-cor}, \ref{lem:sym-arg} and \ref{lem:int-arg} then yields a proof of the relative inverse theorem for the $U^3$-norm (Theorem \ref{thm:u3-inverse}).

\section{Bounding the Ramsey number of Brauer quadruples via sparsity-expansion}\label{sec:ramsey-bound}
In the following sections, we demonstrate how to execute a density increment on quadratic level sets, using our relative $U^3$-inverse theorem (Theorem \ref{thm:u3-inverse}) as a key black box. We have chosen to showcase methods by proving an exponential bound on the Ramsey number of Brauer quadruples (Theorem \ref{thm:exp-brauer}), since it is an application  we have been unable to prove using other means.

We begin by reducing Theorem \ref{thm:exp-brauer} to the following dichotomy. A result of this type appears in work of Chapman and the  author \cite{chapman2020ramsey} on double exponential bounds  for the Ramsey number of  Brauer quadruples, however that result is not relative to quadratic level sets.
\begin{lemma}[Sparsity-expansion dichotomy]\label{lem:brauer-sp-exp-exp}
	 Let $p$ be an odd prime. For any $\alpha, \beta > 0$ and sets $A_1, \dots, A_r \subset \F_p^n$, there exists a subspace $H\leqslant\F_p^{n}$ and quadratic forms $Q = (q_1, \dots, q_d)$ (all defined on $H$) such that 
	$$
	d\ll_p (r/\alpha\beta)^{O_p(1)}, \qquad \codim_{\F_p^n}(H) \ll_p (r/\alpha\beta)^{O_p(1)}
	$$ 
	and for each $A_i$ either of the following hold:
	\begin{itemize}
		\item (Sparsity).  
		\begin{equation}\label{eqnSchurSparse}
			|A_i \cap H \cap Q^{-1}(0)| < \alpha |H\cap Q^{-1}(0)|;
		\end{equation}
		\item (Expansion).  
		\begin{equation}\label{eqnSchurExp}
			|\set{y \in H\cap Q^{-1}(0) : \exists x, x+y, x+2y \in A_i}| > \brac{1 - \beta} |H\cap Q^{-1}(0)|.
		\end{equation}
	\end{itemize} 
\end{lemma}
Theorem \ref{thm:exp-brauer} follows on setting $p = 3$ in the following, which is a short deduction from the sparsity-expansion dichotomy (Lemma \ref{lem:brauer-sp-exp-exp}).

\begin{theorem}[Exponential upper bound for Brauer quadruples in finite vector spaces]\label{thm:exp-brauer-p}
Suppose that there exists an $r$-colouring of $\F_p^n\setminus\set{0}$ with no monochromatic \textbf{Brauer quadruple}
\begin{equation*}
x,\ y, \ x+y,\ x+2y.
\end{equation*}
Then $n \ll_p r^{O_p(1)}$.
\end{theorem}

\begin{proof}
Let $\F_p^n\setminus\set{0} = C_1\cup\dots\cup C_r$ be an $r$-colouring and apply Lemma \ref{lem:brauer-sp-exp-exp} to the colour classes $C_i$, taking $\alpha = \beta = 1/2r$. By the pigeon-hole principle, some $C_i$ satisfies $|C_i\cap H \cap Q^{-1}(0)\setminus\set{0}| \geq \recip{r}|H\cap Q^{-1}(0)\setminus\set{0}|$. Hence, provided that there exists $h \in H\setminus\set{0}$ with $Q(h) = 0$ (so that $|H\cap Q^{-1}(0)|\geq 2$), we deduce that $|C_i \cap H\cap Q^{-1}(0) | \geq \recip{2r} |H\cap Q^{-1}(0)|$. It follows from the sparsity-expansion dichotomy (Lemma \ref{lem:brauer-sp-exp-exp}) and inclusion-exclusion that
\begin{multline*}
|\set{y \in C_i \cap H \cap Q^{-1}(0): \exists x, x+y , x+2y \in C_i}| > \\\brac{1 - \trecip{2r}} |H\cap Q^{-1}(0)| +\trecip{2r} |H\cap Q^{-1}(0)| - |H\cap Q^{-1}(0)| = 0.
\end{multline*}

It follows that we obtain a monochromatic Brauer configuration unless $H \cap Q^{-1}(0) = \set{0}$. By the Chevalley-Warning theorem \cite{ChevalleyWarning}, the latter  implies that $\dim H \leq 2d$. Using our bounds on $d$ and $\codim(H)$,   we obtain a monochromatic Brauer configuration unless $n \ll_p (r/\alpha\beta)^{O_p(1)} \ll_p r^{O_p(1)}$.
\end{proof}

We prove our sparsity-expansion dichotomy (Lemma \ref{lem:brauer-sp-exp-exp}) in \S\ref{sec:sparsity-expansion-density-increment}. The proof uses a density increment argument, a key component of which is our relative $U^3$-inverse theorem paired  with the following, which says that the Brauer counting operator (relativised to a quadratic level set) is controlled by the $U^3$-norm.

\begin{lemma}[$U^3$-control of Brauer on high-rank quadratic level sets]\label{thm:brauer-inverse}
Let $H$ be a finite vector space over $\F_p$ with $p$ odd, let $Q = (q_1, \dots, q_d)$ be quadratic polynomials on $H$ whose rank (as a $d$-tuple) is at least $R$. Write $Q_0$ for the homogeneous quadratic part of $Q$. For any 1-bounded functions $f_0, f_1, f_2, g : H \to \C$ with $\supp(f_i) \subset Q^{-1}(0)$ and $\supp(g) \subset Q_0^{-1}(0)$, if  
\begin{multline}\label{eq:brauer-non-unif}
\Bigabs{\sum_{x,y} f_0(x)f_1(x+y)f_2(x+2y)g(y)} \geq\\
 \eta\Bigabs{\sum_{x,y} 1_{Q^{-1}(0)}(x)1_{Q^{-1}(0)}(x+y)1_{Q^{-1}(0)}(x+2y)1_{Q_0^{-1}(0)}(y)},
\end{multline}
then either $R\ll d$ or for each $f_i$ we have $\normnorm{f_i}_{U^3} \gg \eta \normnorm{1_{Q^{-1}(0)}}_{U^3}$. 
\end{lemma}

The remainder of this section is taken up with proving this result. We begin by estimating the right-hand side of \eqref{eq:brauer-non-unif}.
\begin{lemma}[Counting Brauer quadruples on high-rank level sets]\label{lem:counting-brauer}
Let $H$ be a finite vector space over $\F_p$ with $p$ odd, let $Q = (q_1, \dots, q_d)$ be quadratic polynomials on $H$ whose rank (as a $d$-tuple) is at least $R$. Write $Q_0$ for the homogeneous quadratic part of $Q$.  Then for any  $A \subset Q_0^{-1}(0)$ we have
\begin{equation}\label{eq:number-brauer-quad-level}
\sum_{x, y} 1_{Q^{-1}(0)}(x) 1_{Q^{-1}(0)}(x+y)1_{Q^{-1}(0)}(x+2y) 1_A(y)\\
 = |H||A|p^{-2d} + O\brac{|H|^2p^{-R/2}}.
\end{equation}
\end{lemma}

\begin{proof}
By orthogonality
$$
1_{Q^{-1}(0)}(x) = p^{-d}\sum_{\gamma \in \F_p^d} e_p\brac{\gamma  Q(x)}.
$$ 
Substituting this into our counting operator, the left-hand side of \eqref{eq:number-brauer-quad-level} becomes $p^{-3d}$ multiplied by
\begin{equation}\label{eq:orthogonality-brauer-count}
 \sum_{\gamma_i \in \F_p^d}\sum_{y\in A}\sum_{x \in H} e_p\brac{\gamma_0  Q(x)+\gamma_1  Q(x+y)+\gamma_2  Q(x+2y)}.
 \end{equation}
We may write the inner sum in \eqref{eq:orthogonality-brauer-count} in the form 
$$
\sum_{x \in H}
e_p\brac{\sqbrac{\gamma_0+\gamma_1+\gamma_2} Q(x) + \ell_{y,\gamma}(x) + c_{y,\gamma}},
$$
for some linear forms $\ell_{y,\gamma} : H \to \F_p$ and constants $c_{y,\gamma} \in \F_p$. By the Weyl bound (Lemma \ref{lem:weyl-bound}) and our rank assumption, this inner sum is at most $|H|p^{-R/2}$, unless $\gamma_0 + \gamma_1 + \gamma_2 = 0$. We thus see that \eqref{eq:orthogonality-brauer-count} is within $O\brac{|A||H|p^{3d-R/2}}$ of
\begin{equation}\label{eq:orth-brauer-2}
 \sum_{\gamma_1,\gamma_2\in \F_p^d }\ \sum_{y\in A} \sum_{x \in H} e_p\brac{ \gamma_1\sqbrac{Q(x+y)-Q(x)}+\gamma_2  \sqbrac{Q(x+2y)-Q(x)}}.
\end{equation}

Let $B$ denote the $d$-tuple of symmetric bilinear forms associated to $Q_0$ and let $L$ denote the homogeneous linear part of $Q$. Expanding the quadratic $Q_0(x+y) = Q_0(x) + 2B(x,y) + Q_0(y)$, hence for $y \in Q_0^{-1}(0)$ we have
\begin{equation*}
\gamma_1\sqbrac{Q(x+y)-Q(x)}+\gamma_2  \sqbrac{Q(x+2y)-Q(x)}\\
= (\gamma_1+2\gamma_2)\sqbrac{2B(x, y)+Ly}.
\end{equation*}
Thus \eqref{eq:orth-brauer-2} becomes
\begin{equation*}
 \sum_{\gamma_1,\gamma_2\in \F_p^d }\ \sum_{y\in A} \sum_{x \in H} e_p\brac{ (\gamma_1+2\gamma_2)\sqbrac{2B(x, y)+Ly}}.
 \end{equation*}
 Notice that if $x \mapsto (\gamma_1+2\gamma_2)B(x, y)$ is not the zero linear form on $H$ , then 
 \[
 \sum_{x \in H} e_p\brac{ (\gamma_1+2\gamma_2)\sqbrac{2B(x, y)+Ly}} = 0.
\]
If $\gamma_1+2\gamma_2 \neq 0$ then our rank assumption means that the subspace of $y \in H$ for which $(\gamma_1+2\gamma_2) B(\cdot,y)$ is the zero map $H \to \F_p$ has dimension at most $\dim H-R$. Thus
\begin{equation*}
 \sum_{\gamma_1,\gamma_2\in \F_p^d }\ \sum_{y\in A} \sum_{x \in H} e_p\brac{ (\gamma_1+2\gamma_2)\sqbrac{B(x, y)+Ly}} = p^d |A||H| + O(p^{2d}|H|^2p^{-R}).
 \end{equation*}
The result follows.
\end{proof}

\begin{proof}[Proof of Theorem \ref{thm:brauer-inverse}]
We give the argument for $f_2$, the remaining cases being similar. By Corollary \ref{cor:quad-level-size} and Lemma \ref{lem:counting-brauer}, the right-hand side of \eqref{eq:brauer-non-unif} is $\gg \eta |H|^2p^{-3d}$, or else $R \ll d$. Using this and Cauchy-Schwarz, we have
\begin{multline*}
 \eta |H|^2 p^{-3d} \ll \Bigabs{\sum_{x,y} f_0(x)f_1(x+y)f_2(x+2y)g(y)} \\
 \leq \Bigbrac{\sum_y |g(y)|^2}^{1/2} \Bigbrac{\sum_{y} \Bigabs{\sum_xf_0(x)f_1(x+y)f_2(x+2y)}^2}^{1/2}.
\end{multline*}
Since the support of $g$ lies in $Q_0^{-1}(0)$, we can use Corollary \ref{cor:quad-level-size} to give the estimate
$
\sum_y |g(y)|^2 \ll |H|p^{-d},
$
or else $R \ll d$. Therefore
\begin{multline*}
 \eta^2 |H|^3 p^{-5d} \ll \sum_{y} \Bigabs{\sum_xf_0(x)f_1(x+y)f_2(x+2y)}^2\\
 = \sum_h \sum_{x,y} \Delta_h f_0(x) \Delta_hf_1(x+y)\Delta_hf_2(x+2y)\\
 \leq \Bigbrac{\sum_{h,x} |\Delta_h f_0(x)|^2}^{1/2}\Bigbrac{\sum_{h,x}\Bigabs{\sum_y \Delta_hf_1(x+y) \Delta_h f_2(x+2y)}^2}^{1/2}.
\end{multline*}

As before, since $\supp(f_0) \subset Q^{-1}(0)$, we have $R\ll d$ or 
$$
\sum_{h,x} |\Delta_h f_0(x)|^2 \leq \sum_{h,x} 1_{Q^{-1}(0)}(x) 1_{Q^{-1}(0)}(x+h) = |Q^{-1}(0)|^2 \ll |H|^2p^{-2d}.
$$
Hence
\begin{multline}\label{eq:almost-final-u3-bound}
 \eta^4 |H|^4 p^{-8d} \ll  \sum_{h,x}\Bigabs{\sum_y \Delta_hf_1(x+y) \Delta_h f_2(x+2y)}^2
 = \sum_{h_1, h_2,x,y} \Delta_{h_1,h_2}f_1(x) \Delta_{h_1,2h_2} f_2(y)\\
 \ll \Bigbrac{\sum_{h_1, h_2}\Bigabs{\sum_x\Delta_{h_1,h_2}f_1(x)}^2}^{1/2}\Bigbrac{\sum_{h_1, h_2}\Bigabs{\sum_y \Delta_{h_1,2h_2} f_2(y)}^2}^{1/2}
 = \norm{f_1}_{U^3}^4\norm{f_2}_{U^3}^4.
\end{multline}

We have 
$$
\norm{f_1}_{U^3}^8\leq \sum_{h_1, h_2, h_3} \sum_x\Delta_{h_1, h_2, h_3} 1_{Q^{-1}(0)}(x).
$$
By Lemma \ref{lem:struct-diff}, the difference function
$
 \Delta_{h_1, h_2, h_3} 1_{Q^{-1}(0)} $ equals the indicator function of the set 
 $$
 Q^{-1}(0)\cap (V_{h_1, h_2, h_3} + x_0),
 $$
where $x_0$ is any fixed element of the support, and
 $$
 V_{h_1, h_2, h_3} = \set{x\in H : B(h_1, x) = B(h_2, x) = B(h_3, x) = 0}.
 $$
 
 From Corollary \ref{cor:pseud-thrice-diff}, for all but $ |H|^3p^{3d - R}$ triples $(h_1, h_2, h_3)$ we have
  $$
 |Q^{-1}(0)\cap (V_{h_1, h_2, h_3} + x_0)| = |H|p^{-4d} +O(|H|p^{-R/2}).
 $$
Hence $R\ll d$ or $
 \norm{f_1}_{U^3}^8\leq \normnorm{1_{Q^{-1}(0)}}_{U^3}^8\asymp |H|^4p^{-4d}$. Incorporating this into \eqref{eq:almost-final-u3-bound} gives $
  \eta \norm{1_{Q^{-1}(0)}}_{U^3} \ll \norm{f_2}_{U^3}$, as required.
\end{proof}

\section{Partitioning into high-rank quadratic level sets}\label{sec:part-high}
Notice that in our $U^3$-inverse theorem (Theorem \ref{thm:u3-inverse}), we input a high-rank $d$-tuple of quadratics $Q= (q_1, \dots, q_d)$ and output a $(d+1)$-tuple $(q_1, \dots, q_d, q)$. However, it is not guaranteed  that this new $(d+1)$-tuple is also high-rank. In order to ensure this, it may be necessary to modify the $(d+1)$-tuple and partition the space $H$ into further affine subspaces. The purpose of this section is to demonstrate such a procedure.

\begin{definition}[Maximal level set density]\label{def:max-level-density}
Let $H\leq \F_p^n$ with $p$ an odd prime and let $Q = (q_1, \dots, q_d)$ be homogeneous quadratic forms on $H$. For $A \subset \F_p^n$, define
\begin{equation}\label{eq:max-level-density}
\delta_{Q}(A) := \max_{x,L,a}\frac{|(A-x) \cap (Q+L)^{-1}(a)|}{| (Q+L)^{-1}(a)|},
\end{equation}
where the maximum is taken over all $x \in \F_p^n$, $a \in \F_p^d$ and $d$-tuples $L = (\ell_1, \dots, \ell_d)$ of linear forms on $H$. We emphasise that $(Q+L)^{-1}(a)$ is always a subset of $H$, even if $A-x$ is not, so that the quadratic forms are always accompanied by an implicit subspace $H$.
\end{definition}
\begin{lemma}[Maximal level set density is preserved on increasing complexity]\label{lem:density-preservation}
Let $Q = (q_1, \dots, q_d)$ be a $d$-tuple of homogeneous quadratic forms all defined on the same subspace $H \leq \F_p^n$. Given an additional quadratic form $q : H \to \F_p$, any $A \subset \F_p^n$ satisfies
$$
\delta_{Q,q}(A) \geq \delta_{Q}(A).
$$
\end{lemma}
\begin{proof}
Let $x,L, a$ attain the maximum in \eqref{eq:max-level-density}. Partitioning according the level sets of $q$, we have that
$$
\sum_{b \in \F_p} |(A-x) \cap (Q+L)^{-1}(a) \cap q^{-1}(b)| = \delta_Q(A)  \sum_{b \in \F_p} | (Q+L)^{-1}(a) \cap q^{-1}(b)|.
$$
In particular, there exists $b$ such that
\[
\delta_{Q,q}(A) \geq \frac{|(A-x) \cap (Q+L)^{-1}(a) \cap q^{-1}(b)|}{| (Q+L)^{-1}(a) \cap q^{-1}(b)|}\geq \delta_Q(A). \qedhere
\]
\end{proof}
\begin{lemma}[Partitioning into high-rank level sets]\label{lem:high-rank-partitioning}
Let $Q = (q_1, \dots, q_d)$ be a $d$-tuple  of homogeneous quadratic forms all defined on the same subspace $H \leq \F_p^n$. Given $R\geq 0$ and $A_1, \dots, A_r\subset \F_p^n$, there exists a subspace $\tilde H \leq H$ and homogeneous quadratic forms $\tilde Q = (\tilde q_1, \dots, \tilde q_{\tilde d})$ on $\tilde H$ such that:
\begin{itemize}
\item $\delta_{ \tilde Q}(A_i) \geq \delta_Q(A_i)$ for all $1\leq i\leq r$;
\item the rank of $\tilde Q$ on $\tilde H$ is at least $R$;
\item $\tilde H$ has codimension in $H$ at most $(R+r-1)d$;
\item $\tilde d \leq d$ (in fact, each $\tilde q_i$ coincides with the restriction of some $q_j$ to $\tilde H$).
\end{itemize}
\end{lemma}
\begin{proof}
We proceed by induction on $d \geq 0$. The case $d= 0$ holds since the rank of the empty tuple of quadratic forms is (vacuously) at least $R$. Let us therefore suppose that $d > 0$. If the rank of $(q_1, \dots, q_d)$ on $H$ is at least $R$, then the conclusion of the lemma holds with $\tilde H =H$ and $\tilde Q = Q$. We may therefore assume that there exist $\lambda_1,\dots, \lambda_d\in \F_p$ (not all zero) such that the rank of $\lambda_1q_1+\dots + \lambda_d q_d$ on $H$ is at most $R-1$. Re-labelling indices we may assume that $\lambda_d \neq 0$; in fact (dividing through by $-\lambda_d$) we may assume that $\lambda_d = -1$. Hence there exists a subspace $K \leq H$ of codimension strictly less than $R$ such that on $K$ we have $q_d = \lambda_1q_1 + \dots + \lambda_{d-1}q_{d-1}$.

For each $A_i$, let $x_i, L_i, a_i$ satisfy
$$
\delta_Q(A_i) = \frac{|(A_i-x_i) \cap (Q+L_i)^{-1}(a_i)|}{|  (Q+L_i)^{-1}(a_i)|}.
$$
Partitioning $H$ into cosets of $K$, we see that for each $A_i$ there exists a translate $t_i$ such that
$$
\frac{|(A_i-x_i)\cap (K+t_i) \cap (Q+L_i)^{-1}(a_i)|}{| (K+t_i) \cap (Q+L_i)^{-1}(a_i)|} \geq \delta_Q(A_i).
$$
Notice that there exist $\tilde L_i,\tilde a_i$ such that
$$
(K+t_i) \cap (Q+L_i)^{-1}(a_i) = \sqbrac{K\cap (Q+\tilde L_i)^{-1}(\tilde a_i)} + t_i.
$$
Hence there exists $\tilde x_i$ such that
$$
\frac{|(A_i-\tilde x_i)\cap K \cap (Q+\tilde L_i)^{-1}(\tilde a_i)|}{| K \cap (Q+\tilde L_i)^{-1}(\tilde a_i)|} \geq \delta_Q(A_i).
$$

Given a $d$-tuple $a = (a_1, \dots, a_d)$ write $a':= (a_1, \dots, a_{d-1})$. Since $q_d = \lambda_1q_1 + \dots + \lambda_{d-1}q_{d-1}$ on $K$, one may check that for each $i$ there exists a linear form $\ell_i : K \to \F_p$ and $b_i \in \F_p$ such that
$$
K \cap (Q+\tilde L_i)^{-1}(\tilde a_i) = K \cap (Q'+\tilde L_i')^{-1}(\tilde a_i')\cap \ell_i^{-1}( b_i).
$$
Notice that 
$$
K \cap \ell_i^{-1}(b_i) = K \cap (\ker \ell_i + s_i)
$$
for some $s_i \in K$. Hence
$$
K \cap (Q+\tilde L_i)^{-1}(\tilde a_i) = K\cap (\ker \ell_i + s_i) \cap (Q'+\tilde L_i')^{-1}(\tilde a_i').
$$

Let $\hat K$ denote the subspace
$$
K \cap \ker \ell_1 \cap \dots \cap \ker \ell_r,
$$
so that $\hat K$ has codimension at most $r$ within $K$, hence codimension at most $r + R-1$ within $H$. Partitioning each $K\cap (\ker \ell_i + s_i)$ into cosets of $\hat K$, we see that for each $A_i$ there exists a translate $y_i$ such that
$$
\frac{|(A_i-\tilde x_i)\cap (\hat K + y_i) \cap  (Q'+\tilde L_i')^{-1}(\tilde a_i')|}{| (\hat K + y_i) \cap  (Q'+\tilde L_i')^{-1}(\tilde a_i')|} \geq \delta_Q(A_i).
$$
Re-arranging once more, there exist $\hat x_i,\hat L_i,\hat a_i$ such that
$$
\frac{|(A_i-\hat x_i)\cap \hat K \cap  (Q'+\hat L_i)^{-1}(\hat a_i)|}{| \hat K  \cap  (Q'+\hat L_i)^{-1}(\hat a_i)|} \geq \delta_Q(A_i).
$$
Hence on writing $\hat Q$ for the restriction of $Q'$ to $\hat K$, we deduce that
$$
\delta_{\hat Q}(A_i) \geq \delta_Q(A_i) \qquad (1 \leq i \leq r).
$$

Since $\hat Q$ is a $(d-1)$-tuple of quadratic forms, we may apply the induction hypothesis to yield the claimed result.
\end{proof}
\section{Proving sparsity-expansion via density increment}\label{sec:sparsity-expansion-density-increment}
The aim of this section is to prove the sparsity-expansion dichotomy for Brauer quadruples (Lemma \ref{lem:brauer-sp-exp-exp}). 
We proceed by an iterative procedure.
\begin{iteration}\label{it:only-iteration}
Suppose that at stage $m$ of our iteration we have a subspace  $H = H^{(m)} \leq \F_p^n$ and homogeneous quadratic forms $ Q= Q^{(m)} = (q_1, \dots, q_d) : H \to \F_p^d$ where:
\begin{itemize}
\item $d \leq m$.
\item the rank of $Q$ (as a $d$-tuple) on $H$ is at least $R$.
\item $\codim_{\F_p^n}\brac{H} \leq (R+r-1)m(m+1)/2$.
\item Writing $\delta_Q$ for the maximal level set density (as in Definition \ref{def:max-level-density}), we have 
$$
\sum_{i=1}^r\delta_Q(A_i) \gg_p m(\alpha\beta)^{O_p(1)}.
$$
\end{itemize}
\end{iteration}
We initiate this iteration on taking $H^{(0)} := \F_p^n$ and taking $Q^{(0)}$ to be the empty tuple.  

Suppose that our iteration has reached stage $m$. Given $H = H^{(m)}$ and $Q = Q^{(m)}$ as in Iteration \ref{it:only-iteration}, we classify each $A_i$ according to the following:
\begin{itemize}
\item  (Sparse sets).  $A_i$ is \emph{sparse} if the maximal level set density (Definition \ref{def:max-level-density}) satisfies
\begin{equation}\label{eq:iteration-sparse}
\delta_Q(A_i) < \alpha;
\end{equation}
\item  (Dense expanding sets).  $A_i$ is \emph{dense expanding} if $\delta_Q(A_i) \geqslant \alpha$ and we have the expansion estimate
\begin{equation}\label{eq:iteration-expansion}
\abs{\set{y \in H\cap Q^{-1}(0): \exists x, x+y, x+2y \in A_i}} > \brac{1 - \beta} \abs{H\cap Q^{-1}(0)};
\end{equation}
\item  (Dense non-expanding sets).  $A_i$ is \emph{dense non-expanding} if it is neither sparse as in \eqref{eq:iteration-sparse} nor dense expanding as in \eqref{eq:iteration-expansion}, so that
\begin{equation}\label{eq:dense-non-exp}
\delta_Q(A_i) \geq \alpha \quad \text{and}\quad  \frac{\abs{\set{y \in H\cap Q^{-1}(0): \exists x, x+y, x+2y \in A_i}}}{\abs{H\cap Q^{-1}(0)}} \leq \brac{1 - \beta} .
\end{equation}
\end{itemize}
If there are no dense non-expanding $A_i$, then the dichotomy claimed in Lemma \ref{lem:brauer-sp-exp-exp} is satisfied, and we terminate our iteration.  Let us show how the existence of a dense non-expanding $A_i$ allows the iteration to continue.

By the definition of maximal level set density (Definition \ref{def:max-level-density}), there exists $t,L,a$ such that
$$
|(A_i-t)\cap (Q+L)^{-1}(a)| = \delta_Q(A_i) |(Q+L)^{-1}(a)|.
$$ 
We define dense subsets $A \subset (Q+L)^{-1}(a)$ and $B\subset  Q^{-1}(0)$ by taking:
\begin{equation}\label{schur A B defn}
\begin{split}
A & := (A_i - t) \cap  (Q+L)^{-1}(a);\\ 
 B & := \set{y \in Q^{-1}(0) : \nexists x, x+y, x+2y \in A_i}. 
 \end{split}
\end{equation}
Our dense non-expanding assumption implies that $|A| \geq  \alpha |(Q+L)^{-1}(a)|$ and that $ |B| \geqslant \beta| Q^{-1}(0)|$.  Moreover, it follows from our construction \eqref{schur A B defn} that 
\begin{equation}\label{eq:small-conv-inner-prod}
\sum_{x, y} 1_A(x) 1_A(x+y)1_A(x+2y) 1_B(y) = 0.
\end{equation}
We compare this with the count which replaces each occurrence of $1_A$ with $\delta 1_{(Q+L)^{-1}(a)}$, where $\delta := \delta_Q(A_i)$. By a standard telescoping argument, there exist functions $f_0, f_1, f_2 : H \to [-1,1]$ with support contained in $(Q+L)^{-1}(a)$, at least one of which is equal to $1_A - \delta 1_{(Q+L)^{-1}(a)}$, and such that
\begin{multline}\label{eq:large-count}
\Bigabs{\sum_{x, y} f_0(x) f_1(x+y)f_2(x+2y) 1_B(y)} \gg\\ \delta^3\sum_{x, y} 1_{(Q+L)^{-1}(a)}(x) 1_{(Q+L)^{-1}(a)}(x+y)1_{(Q+L)^{-1}(a)}(x+2y) 1_B(y)
\end{multline}

By Lemma \ref{lem:counting-brauer} and Corollary \ref{cor:quad-level-size}, the right-hand side of \eqref{eq:large-count} is $\gg \delta^3\beta |H|^2p^{-3d}$, or else $R \ll d + \log(2/\beta)$. Hence we may assume that 
\begin{multline*}
\Bigabs{\sum_{x, y} f_0(x) f_1(x+y)f_2(x+2y) 1_B(y)} \gg \alpha^3\beta |H|^2p^{-3d}\\
\gg \alpha^3\beta\sum_{x, y} 1_{(Q+L)^{-1}(a)}(x) 1_{(Q+L)^{-1}(a)}(x+y)1_{(Q+L)^{-1}(a)}(x+2y) 1_{Q^{-1}(0)}(y) .
\end{multline*}
By $U^3$-control of the Brauer counting operator (Lemma \ref{thm:brauer-inverse}), we deduce that either $R \ll d + \log(2/\beta)$ or else 
$$
\normnorm{1_A - \delta 1_{(Q+L)^{-1}(a)}}_{U^3} \gg \alpha^3\beta \norm{1_{(Q+L)^{-1}(a)}}_{U^3}.
$$

Applying the $U^3$-inverse theorem (Theorem \ref{thm:u3-inverse}), we  see that either $R \ll_p d + \log(2/\alpha\beta)$ or there exists a quadratic form $q$ and linear form $\ell$ such that 
\begin{equation*}
\Bigabs{\sum_{x\in (Q+L)^{-1}(a)} \sqbrac{1_A(x) -  \delta}e_p\sqbrac{q(x)+\ell(x)}}
\\ \gg_p (\alpha \beta)^{O_p(1)} | (Q+L)^{-1}(a)|.
\end{equation*}
Partitioning into the level sets of $q+l$ and using the triangle inequality gives
\begin{equation*}
\sum_{b \in \F_p}\Bigabs{\sum_{x\in(Q+L)^{-1}(a)\cap (q+l)^{-1}(b)} \sqbrac{1_A(x) -  \delta}}
\\ \gg_p (\alpha \beta)^{O_p(1)} \sum_{b \in \F_p}|(Q+L)^{-1}(a)\cap (q+l)^{-1}(b)|.
\end{equation*}
Since $1_A - \delta$ has mean zero on $(Q+L)^{-1}(a)$, we have that 
\begin{equation*}
\sum_{b \in \F_p}\ \sum_{x\in(Q+L)^{-1}(a)\cap (q+l)^{-1}(b)} \sqbrac{1_A(x) -  \delta} = 0.
\end{equation*}
Adding this to our inequality gives
\begin{equation*}
\sum_{b \in \F_p} \Bigbrac{\ \sum_{x\in(Q+L)^{-1}(a)\cap (q+l)^{-1}(b)} \sqbrac{1_A(x) -  \delta}}_+
\\ \gg_p (\alpha \beta)^{O_p(1)} \sum_{b \in \F_p}|(Q+L)^{-1}(a)\cap (q+l)^{-1}(b)|,
\end{equation*}
where $x_+ := \max\set{x, 0} = \trecip{2}(x + |x|)$ denotes the positive part of a real number $x$. Hence by the pigeon-hole principle, there exists $b\in \F_p$ with 
$$
\frac{|A\cap (Q+L)^{-1}(a)\cap (q+l)^{-1}(b)|}{|(Q+L)^{-1}(a)\cap (q+l)^{-1}(b)|} \geq \delta + \Omega_p\brac{\alpha\beta}^{O_p(1)} .
$$
Recalling that $\delta = \delta_Q(A_i)$ and that for some $t$ we have
$$
A= (A_i-t)\cap (Q+L)^{-1}(a),
$$
we deduce that
$$
\delta_{Q,q}(A_i) \geq \delta_Q(A_i) + \Omega_p\brac{\alpha\beta}^{O_p(1)}.
$$
In other words, one of our sets $A_i$ has maximal level set density which increments on adding the additional quadratic form $q$ to $Q$.

Since maximal level set density does not decrease on increasing the number of quadratic forms (by Lemma \ref{lem:density-preservation}), the sum of the maximal level set densities also increments:
$$
\sum_{j=1}^r \delta_{Q,q}(A_j) \geq \Omega_p\brac{\alpha\beta}^{O_p(1)} +\sum_{j=1}^r \delta_{Q}(A_j)
$$

By Lemma \ref{lem:high-rank-partitioning}, there exists a subspace $\tilde H \leq H$ and quadratic forms $\tilde Q = (\tilde q_1, \dots, \tilde q_{\tilde d})$ on $\tilde H$ such that:
\begin{itemize}
\item $\delta_{ \tilde Q}(A_j) \geq \delta_{Q,q}(A_j)$ for all $1\leq j\leq r$;
\item the rank of $\tilde Q$ (as a $d$-tuple) on $\tilde H$ is at least $R$;
\item $\tilde H$ has codimension in $H$ at most $(R+r-1)(d+1)$;
\item $\tilde d \leq d+1$.
\end{itemize}

We have therefore established that, given the data in Iteration \ref{it:only-iteration}, at least one of three possibilities occurs: there are no dense non-expanding sets $A_i$ (those sets satisfying \eqref{eq:dense-non-exp}); or the rank is too small $R \ll_p m + \log(2/\alpha\beta)$; or the assumptions of the iteration scheme (Iteration \ref{it:only-iteration}) remain valid with $m$ replaced by $m+1$. 

Since densities are always at most 1, at every stage of Iteration \ref{it:only-iteration} we have
$$
m \brac{\alpha\beta}^{O_p(1)} \ll_p \sum_{j=1}^r \delta_Q(A_j) \leq r.
$$
Hence the iteration must terminate at some $m \ll_p r\brac{\alpha\beta}^{-O_p(1)}$. Taking our rank $R$ to be of the form
$$
R := C_pr\brac{\alpha\beta}^{-C_p}
$$
for some sufficiently large absolute constant $C_p$, we see that the termination must have occurred because there are no dense non-expanding sets $A_i$.  Incorporating this value of $m$ and $R$ into the assumptions of Iteration \ref{it:only-iteration}, we deduce that we have the conclusion of the sparsity-expansion dichotomy (Lemma \ref{lem:brauer-sp-exp-exp}) with  $d \ll_p r\brac{\alpha\beta}^{-O_p(1)}$ and  $\codim_{\F_p^n}(H) \ll_p (r/\alpha\beta)^{O_p(1)}$. This completes the proof of Lemma \ref{lem:brauer-sp-exp-exp}.

\appendix

\section{Density bounds for  configurations of complexity two}\label{sec:density}
The purpose of this appendix is to demonstrate how one can avoid running density increment arguments with respect to quadratic level sets if one wishes to obtain polylogarithmic bounds on the density of sets lacking certain translation-invariant configurations of complexity two. In particular, we prove Theorem \ref{thm:poly-log-density}. Our argument combines the polynomial method of Croot-Lev-Pach with  a weak quadratic regularity lemma of Green-Tao. 

\subsection{Proof of Theorem \ref{thm:poly-log-density}}
Let us first state the required application of the polynomial method.

\begin{lemma}[Polynomial method]\label{lem:poly}
Let $C_0, C_1, C_2, C_3 \in \F_p\setminus\set{0}$ with $C_0 + C_1 + C_2 + C_3 = 0$. Then for any function $f : \F_p^n \to [0,1]$ with $\sum_{x\in \F_p^n} f(x) \geq \delta$, we have
$$
\sum_{C_0x_1 + C_1x_2 + C_3x_3 + C_4x_4 = 0} f(x_0)f(x_1)f(x_2)f(x_3) \gg_p \delta^{O_p(1)} p^{3n}.
$$
\end{lemma}

\begin{proof}
Let $A := \set{x \in \F_p^n : f(x) \geq \trecip{2}\delta }$, so that $|A|\geq \trecip{2} \delta p^n$. Define $A_i := \set{C_i x : x \in A}$. If the number of quadruples $(x_0, x_1, x_2, x_3) \in A_0 \times A_1 \times A_2 \times A_3$ such that $x_1 + x_2 + x_3 + x_4 = 0$ is at most $\eta p^{3n}$ where $0 < \eta < 1/2$, then by the arithmetic $4$-cycle removal lemma \cite[Theorem 1.3]{fox2018polynomial}, we can delete at most $\eta^{\Omega_p(1)} p^n$ elements from each $A_i$ so that no such quadruples remain. As a consequence, we can delete at most $4\eta^{\Omega_p(1)} p^n$ elements from $A$ to obtain a set $A'$ for which there are no quadruples $(x_0, x_1, x_2, x_3) \in (A')^4$ with 
\begin{equation}\label{eq:trans-inv}
C_0x_0 + C_1 x_1 + C_2x_2 + C_4 x_4 = 0.
\end{equation} It follows that $A'$ must be empty, for any element $x \in A'$ yields such a quadruple on taking $x_i = x$ for all $i$ (since the $C_i$ sum to zero). Thus we have a contradiction unless $\trecip{2}\delta - 4\eta^{\Omega_p(1)} \leq 0$. 

It follows that the number of solutions to \eqref{eq:trans-inv} with all $x_i \in A$ is $\gg_p \delta^{O_p(1)} p^{3n}$. Thus
\begin{multline*}
\sum_{C_0x_1 + C_1x_2 + C_3x_3 + C_4x_4 = 0} f(x_0)f(x_1)f(x_2)f(x_3)\\ \geq (\delta/2)^4 \sum_{C_0x_1 + C_1x_2 + C_3x_3 + C_4x_4 = 0} 1_A(x_0)1_A(x_1)1_A(x_2)1_A(x_3)\gg_p \delta^{O_p(1)} p^{3n}.\qedhere
\end{multline*}
\end{proof}

We also utilise a weak quadratic regularity lemma of Green and Tao \cite[Theorem 4.10]{GreenTaoNewIa}, which requires the following notation.
\begin{definition}[Projection onto quadratic level sets]
Let $H$ be a finite vector space over $\F_p$ and let $Q$ be a $d$-tuple of quadratic polynomials on $H$. Given $f : H \to \C$, define the projection $\Pi_Qf : H \to \C$ of $f$ onto the level sets of $Q$ by
$$
\Pi_Q f(x) = \sum_{a \in \F_p^n} 1_{ Q^{-1}(a)}(x) \E_{y \in Q^{-1}(a)} f(y).
$$
Let us also write $f_Q(a)$ for $\E_{y \in Q^{-1}(a)} f(y)$, so that $\Pi_Qf (x) = f_Q(a)$ when $x \in Q^{-1}(a)$. We set $f_Q(a) = 0$ if $Q^{-1}(a) = \emptyset$.
\end{definition}

\begin{theorem}[Weak regularity, \cite{GreenTaoNewIa} Theorem 4.10]\label{thm:weak-reg}
Let $A \subset \F_p^n$ be a set with density $\delta := |A|p^{-n}$. Let $0< \eta, \eps < 1/2$ and  $R\geq 1$. Then there exists a subspace $H \leq \F_p^n$ of codimension at most $(\eps\eta)^{-O_p(1)} R$, there exists a translate $t \in \F_p^n$ such that the density of $A-t$ on $H$ is at least $\delta -\eps$, and there exists a $d$-tuple $Q$ of quadratic polynomials on $H$ of rank at least $R$ such that $d \leq (\eps\eta)^{-O_p(1)}$ and
\begin{equation}\label{eq:weak-regularity}
\norm{1_{(A-t)\cap H} - \Pi_Q1_{(A-t)\cap H}}_{U^3} \leq \eta \norm{1_H}_{U^3}.
\end{equation} 
\end{theorem}

\begin{proof}[Proof of Theorem \ref{thm:poly-log-density}]
Let $\delta := |A|p^{-n}$ denote the density of $A$. We apply Green and Tao's weak regularity lemma (Theorem \ref{thm:weak-reg}) with values of $\eta$, $\eps$, $R$ to be determined. Write $B := (A-t)\cap H$ and $f := \Pi_Q1_B$, so that $f(x) = f_Q(a) = \E_{y \in Q^{-1}(a)} 1_B(y)$ when $Q(x) = a$. By Corollary \ref{cor:quad-level-size}, either $R \ll (\eps\eta)^{-O_p(1)}$ or for each $a \in \F_p^d$ we have $|Q^{-1}(a)| \leq 2 |H|p^{-d}$. Since $|B| \geq (\delta-\eps)|H|$, we therefore have that
\begin{multline*}
\sum_{a \in \F_p^d} f_Q(a) = \sum_{a \in \F_p^d}  \E_{y \in Q^{-1}(a)} 1_B(y) = \sum_{a \in \F_p^d} \recip{ |Q^{-1}(a)|}\sum_{y \in Q^{-1}(a)} 1_B(y) \\\geq \trecip{2} |H|^{-1}p^{d}\sum_{a \in \F_p^d} \sum_{y \in Q^{-1}(a)} 1_B(y) 
\geq \trecip{2}(\delta - \eps) p^d . 
\end{multline*}
Taking $\eps = \trecip{2} \delta$, we get that $R \ll (\delta\eta)^{-O_p(1)}$ or $\sum_a f_Q(a) \gg \delta p^d$ (so that $f_Q$ has density $\Omega(\delta)$ on $\F_p^d$).

We may assume that $c_1$, $c_2$, $c_3$ in the statement of Theorem \ref{thm:poly-log-density} are distinct and non-zero, since simpler methods suffice otherwise. It then follows from the usual Cauchy-Schwarz argument (``generalised von Neumann theorem'', see \cite[Lemma 3.1]{GreenTaoNewIa}) plus telescoping identity that
\begin{multline}\label{eq:struct-vs-full}
|\E_{x,y \in H} 1_B(x)1_B(x+c_1y)1_B(x+c_2y)1_B(x+c_3y)-\\ \E_{x,y \in H} f(x)f(x+c_1y)f(x+c_2y)f(x+c_3y)|
 \ll \normnorm{1_B -f}_{U^3}/\normnorm{1_H}_{U^3} \leq \eta.
\end{multline}
Let us therefore analyse the count
\begin{multline*}
\E_{x,y\in H}f(x)f(x+c_1y)f(x+c_2y)f(x+c_3y) =\\ \sum_{a_i \in \F_p^d}\prod_if_Q(a_i)\E_{x,y\in H}1_{ Q^{-1}(a_0)}(x)1_{ Q^{-1}(a_1)}(x+c_1y)1_{ Q^{-1}(a_2)}(x+c_2y)1_{ Q^{-1}(a_3)}(x+c_3y).
\end{multline*}
We decompose the inner sum using orthogonality:
\begin{multline*}
\E_{x,y\in H}1_{ Q^{-1}(a_0)}(x)1_{ Q^{-1}(a_1)}(x+c_1y)1_{ Q^{-1}(a_2)}(x+c_2y)1_{ Q^{-1}(a_3)}(x+c_3y) = 
\E_{\gamma_i \in \F_p^d} \E_{x,y\in H}\\ e_p(\gamma_0\sqbrac{Q(x)-a_0})e_p(\gamma_1\sqbrac{Q(x+c_1y)-a_1})e_p(\gamma_2\sqbrac{Q(x+c_2y)-a_2})e_p(\gamma_3\sqbrac{Q(x+c_3y)-a_3}).
\end{multline*}
For fixed $\gamma_i$, the Weyl bound (Lemma \ref{lem:weyl-bound}) gives that the inner sum is at most $p^{-R/2}$ unless all of the following hold
\begin{align*}
\gamma_0 +\gamma_1 + \gamma_2 + \gamma_3 & = 0,\\
c_1\gamma_1 +c_2\gamma_2 +  c_3\gamma_3  & = 0,\\
c_1^2\gamma_1 +c_2^2\gamma_2 + c_3^2\gamma_3  & = 0.
\end{align*}
Equivalently
\begin{align*}
\gamma_1  = -\tfrac{c_2c_3}{(c_1-c_2)(c_1-c_3)}\gamma_0, \qquad
\gamma_2    = \tfrac{c_1c_3}{(c_1-c_2)(c_2-c_3)}\gamma_0,\qquad
  \gamma_3   = - \tfrac{c_1c_2}{(c_1-c_3)(c_2-c_3)}\gamma_0.
\end{align*}
Therefore, writing $C_0:= 1$, $C_1 := -\tfrac{c_2c_3}{(c_1-c_2)(c_1-c_3)}$, $C_2 := \tfrac{c_1c_3}{(c_1-c_2)(c_2-c_3)}$ and $C_3 :=- \tfrac{c_1c_2}{(c_1-c_2)(c_2-c_3)}$ we have
\begin{multline*}
\E_{x,y\in H}f(x)f(x+c_1y)f(x+c_2y)f(x+c_3y) =\\ p^{-3d}\sum_{a_i \in \F_p^d}\prod_if_Q(a_i)\E_{\gamma \in \F_p^d}e_p\brac{-\gamma\sqbrac{ C_0a_0+C_1a_1+C_2a_2+C_3a_3}}+ O(p^{O(d)-\Omega(R)})\\
= p^{-3d}\sum_{C_0a_0+C_1a_1+C_2a_2+C_3a_3 = 0}f_Q(a_0)f_Q(a_1)f_Q(a_2)f_Q(a_3)+ O(p^{O(d)-\Omega(R)}).
\end{multline*}

Since the $C_i$ sum to zero and $f_Q$ has density $\Omega(\delta)$ on $\F_p^d$ (or else $R \ll (\delta\eta)^{-O_p(1)}$), we can apply the polynomial method (Lemma \ref{lem:poly}) to deduce that
\begin{equation}\label{eq:struct-count}
\E_{x,y\in H}f(x)f(x+c_1y)f(x+c_2y)f(x+c_3y) \geq \Omega_p(\delta^{O_p(1)}) - O(p^{O(d)-\Omega(R)}).
\end{equation}
Thus the left-hand side of \eqref{eq:struct-count} is $\gg_p \delta^{O_p(1)}$ or else $R \ll_p d+ (\delta\eta)^{-O_p(1)}\ll_p(\delta\eta)^{-O_p(1)}$. Using \eqref{eq:struct-vs-full}, we can take $\eta\gg_p \delta^{O_p(1)}$ and $R \ll_p  \delta^{-O_p(1)}$ and thereby guarantee that
$$
\sum_{x,y\in H}1_B(x)1_B(x+c_1y)1_B(x+c_2y)1_B(x+c_3y) \gg_p \delta^{O_p(1)}|H|^2.
$$
Since the number of configurations in $B$ with $y = 0$ is at most $|H|$, it follows that either $B$ (and hence $A$) contains a non-trivial configuration, or else $|H| \ll_p \delta^{-O_p(1)}$. Since the weak regularity lemma (Theorem \ref{thm:weak-reg}) gives (with our choices of $\eps$, $\eta$, $R$) that the codimension of $H$ in $\F_p^n$ is $ \ll_p \delta^{-O_p(1)}$, we deduce that the only way that $A$ can lack non-trivial configurations is if $n \ll_p \delta^{-O_p(1)}$.
\end{proof}

\subsection{Translation-invariant systems of complexity two}

In this subsection, we discuss the limitations of generalising the above approach to all translation-invariant systems of complexity two. In particular, Conjecture \ref{conj:complex-2} seems out of reach of the methods of this appendix, yet may be within reach of the density increment method of \S\S\ref{sec:ramsey-bound}--\ref{sec:sparsity-expansion-density-increment}.

\begin{definition}[Translation-invariant system]\label{def:translation-invariant}
Given vectors $\phi_1, \dots, \phi_k \in \F_p^d$, the \textit{system} $\Phi:= [\phi_1, \dots, \phi_k]$ defines a \textit{configuration} 
$$
\Phi(x_1, \dots, x_d) := \sqbrac{\phi_1^T(x_1, \dots, x_d), \dots, \phi_k^T(x_1, \dots, x_d)}.
$$
For example, the Brauer configuration $\sqbrac{x, x+y, x+2y, y}$ corresponds to the system $\Phi = \sqbrac{(1,0), (1,1), (1,2), (0,1)}$. The {system}  is \textit{translation-invariant} if the space of  configurations contains the diagonal, so that
$$
\set{\Phi(x_1, \dots, x_d) : x_i \in \F_p} \supset \set{[x,x, \dots, x] : x \in \F_p}.
$$
Equivalently, the matrix $\Phi$ has $(1,1, \dots, 1)$ in its row-space, so that there exist $\lambda_1$, \dots, $\lambda_d$ such that for each $i=1, \dots, k$ we have $\sum_{j=1}^d \lambda_j\phi_{ij} =1$. For example, the system corresponding to four-term progressions $\Phi = [(1,0),(1,1),(1,2),(1,3)]$ is translation-invariant, but the system corresponding to the Brauer configuration is not.
\end{definition}

\begin{definition}[Complexity at most two]\label{def:complex-two}

Given vectors $\phi_1, \dots, \phi_k \in \F_p^d$, we say that the {system} $\Phi:= [\phi_1, \dots, \phi_k]$ has 
\textit{complexity at most two} if the $k$ vectors
$$
\phi_i^{\otimes 3} = \phi_i\otimes\phi_i\otimes \phi_i = (\phi_{ij_1}\phi_{ij_2}\phi_{ij_3})_{1\leq j_1, j_2, j_3 \leq d} \qquad (1\leq i \leq k)
$$
are linearly independent over $\F_p$. For example, the Brauer system $\Phi = \sqbrac{(1,0), (1,1), (1,2), (0,1)}$ has complexity 
at most two since
$$
\Phi^{\otimes 3} = [(1,0,0,0,0,0,0,0),(1,1,1,1,1,1,1,1),(1,2,2,4,2,4,4,8),(0,0,0,0,0,0,0,1)]
$$
are linearly independent. 
\end{definition}
%
%
%

%
%

The key bridge between quadratic Fourier analysis and the study of complexity two configurations is the following  result of Manners \cite[Theorem 1.1.5]{manners2021true}.
\begin{theorem}[$U^3$-control of complexity-two configurations, \cite{manners2021true}]\label{thm:true}
Let $p$ be prime and $\phi_1, \dots, \phi_k \in \F_p^d$ such that the system $\Phi = [\phi_1, \dots, \phi_k]$ has complexity at most two. Let $H$ be a finite vector space over $\F_p$, let $f_1, \dots, f_k : H \to \C$ be 1-bounded functions and $0< \eta \leq 1/2$.  Suppose that
\begin{equation*}
\Bigabs{\sum_{x_1, \dots, x_d \in H} \prod_{i=1}^k f_i\sqbrac{\phi_i^T(x_1, \dots, x_d)}} \geq \eta |H|^d.
\end{equation*} 
Then for each $f_i$ we have $\norm{f_i}_{U^3} \geq \eta^{O_\Phi(1)}\norm{1_H}_{U^3}$.
\end{theorem}

In order to approach Conjecture \ref{conj:complex-2} using the method of the previous subsection, we first  use the weak regularity lemma (Theorem \ref{thm:weak-reg}) to pass to a subspace $H$ where we have the approximation \eqref{eq:weak-regularity}. By Manners' $U^3$-control result (Theorem \ref{thm:true}), in order to count configurations in $A$ it suffices to count configurations with the weights $\Pi_Q1_{(A-t)\cap H}$. By our high-rank assumption on $Q$, this reduces to estimating
$$
\sum_{a_1, \dots, a_k \in \F_p^d: c_1^Ta = \dots = c_r^Ta =0} f_Q(a_1)\dots f_Q(a_k),
$$
where $f_Q : \F_p^d \to [0,1]$ is a function satisfying $\E(f_Q) \gg \delta$ and $c_1, \dots, c_r \in \F_p^k$ are a basis for the space of $ (\lambda_1, \dots, \lambda_k)$ satisfying 
$$
\lambda_1\phi_1^{\otimes 2} + \dots +\lambda_k \phi_k^{\otimes 2}= 0.
$$
The system $c_1^Ta = \dots = c_r^Ta =0$ is itself translation-invariant, as demonstrated in the following.
\begin{lemma}
Let $\Phi = (\phi_1, \dots, \phi_k)$ be a translation-invariant system. If there exist $\mu_i \in \F_p$ such that $\sum_{i=1}^k \mu_i\phi_i^{\otimes 2} = 0$ then $\sum_{i=1}^k \mu_i = 0$. 
\end{lemma}

\begin{proof}
Viewing the $\phi_i$ as column vectors, the idea is to show that the matrix $\Phi^{\otimes 2} = (\phi_1^{\otimes 2} , \dots, \phi_k^{\otimes 2})$ has $(1, \dots, 1)$ in its row-space. Since $\mu = (\mu_1, \dots, \mu_k)$ lies in the kernel of $\Phi^{\otimes 2}$ (because $\Phi^{\otimes 2} \mu = 0$), we must therefore have that $(1,\dots, 1)\cdot \mu = 0$.

By translation-invariance, there exist $\lambda_j$ such that for each $i$ we have $\sum_{j=1}^d \lambda_j\phi_{ij} =1$. Squaring this identity gives
$$
1 = \Bigbrac{\sum_{j=1}^d \lambda_j\phi_{ij}}^2 = \sum_{j_1, j_2} \lambda_{j_1}\lambda_{j_2}\phi_{ij_1}\phi_{ij_2}.
$$
Since $\phi_i^{\otimes 2}= (\phi_{ij_1}\phi_{ij_2})_{1\leq j_1, j_2\leq d}$, we have shown that the all-ones vector $(1,\dots, 1)$ lies in the row-space of the matrix $(\phi_1^{\otimes 2}, \dots, \phi_k^{\otimes 2})$.
\end{proof}

One might hope that (just as in the proof of Theorem \ref{thm:poly-log-density}), the configuration determined by the system of equations $c_1^Ta = \dots = c_r^Ta =0$ is attackable with the polynomial method. Gijswijt \cite{gijswijt2023excluding} has general results in this direction, but I expect these may not be general enough for this purpose.



%
\end{document}